\documentclass[10pt]{amsart}
\usepackage{fullpage}
\usepackage[toc,page,title,titletoc,header]{appendix}
\usepackage{amsthm,amsmath,amsfonts,amssymb,euscript,hyperref,graphics,color,slashed,mathrsfs,graphicx,mathtools,enumerate,cite,subfigure}
\usepackage[graphicx]{realboxes}
\hypersetup{colorlinks, citecolor=black, urlcolor=black, linkcolor=black}
\usepackage[english]{babel}
\usepackage{bm}

\allowdisplaybreaks

\numberwithin{equation}{section}

\newtheorem{theorem}{Theorem}[section]
\newtheorem{lemma}[theorem]{Lemma}
\newtheorem{proposition}[theorem]{Proposition}
\newtheorem{corollary}[theorem]{Corollary}
\newtheorem{definition}[theorem]{Definition}

\newtheorem{remark}[theorem]{Remark}

\begin{document}

	\title{Pointwise boundary estimates  for fully nonlinear elliptic equations with nonzero Dirichlet boundary conditions}

	\author{Mengni Li}	
	\address{Southeast University, No.2 SEU Road, Nanjing 211189, P.R. China}
	\email{krisymengni@163.com }

		\author{Chaofan Shi}	
	\address{Southeast University, No.2 SEU Road, Nanjing 211189, P.R. China}
	\email{220231952@seu.edu.cn}

	\thanks{The authors are grateful to Prof. You Li for many discussions on the case of zero boundary conditions and for valuable suggestions to improve the manuscript. This work was partially supported by  National Natural Science Foundation of China (No. 12201107, No. 12301263) and Jiangsu Provincial Scientific Research Center of Applied Mathematics (No.  BK20233002).}

	\begin{abstract} 
		In this paper, we investigate  
		boundary estimates for the Dirichlet problem for a class of fully nonlinear elliptic equations with general   boundary conditions, including   nonzero boundary conditions. 
		Given specific structural conditions on  the problem, we develop  pointwise boundary  upper and lower bound estimates for convex  solutions   based on the subsolution and supersolution method. The global H\"older regularity can be derived as a direct consequence of these pointwise boundary estimates. These results  fundamentally hinge on careful descriptions of the convexity properties of both the domains and the  functions involved.  Moreover, previous results on Monge-Amp\`ere equations with nonzero Dirichlet boundary conditions  can be regarded as a special case of our results.
		
		\noindent
		\textbf{Keywords:} Dirichlet problem, fully nonlinear elliptic equations, nonzero boundary conditions, pointwise boundary estimates, global H\"older estimates
		
		\noindent
		\textbf{MSC 2020:} 35J60, 35B65, 53A15 
	\end{abstract}

	\maketitle
	\tableofcontents

	\section{Introduction}
	\label{sec1} 
	This paper focuses on  
	the following Dirichlet problem for a class of 
	fully nonlinear elliptic partial differential equations:
	\begin{align}
		F(\lambda_{1}(D^{2}u),\cdot \cdot \cdot ,\lambda_{n}(D^{2}u)) &= f(x,u,Du)\quad \text{in} \; \varOmega, \stepcounter{equation}\tag{\theequation}\label{eq01} \\
		u &= \varphi (x)\quad \text{on} \; \partial\varOmega. \stepcounter{equation}\tag{\theequation}\label{eq01boundary}
	\end{align}
	Here, $ \varOmega \subseteq \mathbb{R} ^{n}(n\geq 2) $ is a bounded convex domain, $u:\varOmega\rightarrow \mathbb{R}$ is a convex function,\
	$ \lambda_{1}(D^{2}u),\cdots ,\lambda_{n}(D^{2}u) $ are the eigenvalues of the Hessian matrix $D^{2}u$, and 
	$\varphi:\partial\varOmega\rightarrow \mathbb{R}$ is a continuous function.
	
	\subsection{Connection with geometry}
	
	This problem is deeply interconnected with  geometry.  
	Notably, the equation \eqref{eq01} is of Monge-Amp\`ere type
	when we set the operator   $$F(\lambda_{1}(D^{2}u),\cdot \cdot \cdot ,\lambda_{n}(D^{2}u)) =\det D^2u.$$ As the function $f(x,z,q)$ takes different forms, the Monge-Amp\`ere type equations have various geometric applications \cite{Trudinger-Wang,Figalli,Fefferman}, which are often associated with the curvature of surfaces and geometric properties of manifolds. For $$f(x,z,q)=g(x)(1+|q|^{2})^{\frac{n+2}{2}},$$ the equation is known as the prescribed Gauss curvature equation \cite{Trudinger-Urbas}. The key to solving this equation is directly linked to finding a convex surface defined by $u(x)$ with a specified Gauss curvature $g(x)$. 
	A classical example is when $g(x)$ is a constant function, leading to surfaces of constant Gauss curvature, such as spheres. With $$f(x,z,q) = g(x)|z|^{p-1},$$ the equation relates to the $L_p$-Minkowski problem \cite{Chou-Wang}, which seeks to determine a convex body given its volume and surface area  
	under the $L_p$ norm. In particular, the centroaffine Minkowski problem  
is restricted to the case $p=-n-1$ with the volume considered  
	under centroaffine transformations.  
	When $$f(x,z,q) = |z|^{-(n+2)},$$ the Legendre transform of the solution yields  a complete hyperbolic affine sphere \cite{Calabi,Cheng-Yau}. The study of this case also connects to the properties of metrics in complex geometry, such as K\"ahler metrics, Hilbert metrics, and Poincar\'e metrics \cite{Cheng-Yau-metric,Loewner-Nirenberg,Fefferman}.

	\subsection{Previous boundary estimates results}

	The boundary regularity theory of fully nonlinear elliptic equations,  with various assumptions relative to $\varOmega$, $F$, $f$ and $\varphi$, has drawn considerable attention over the past half century. It is impossible to list all the publications in this field. The readers are referred to  \cite{Caffarelli-Nirenberg-Spruck-1,Caffarelli-Nirenberg-Spruck-3,Cheng-Yau,Cui-Jian-Li,Figalli,Ivochkina,Jian-Li,Jian-Wang-Adv,Jian-Wang-JDG,Krylov,Lazer-McKenna,Li-2021,Li-Li-F,Li-Li-N,Lin-Wang,Tian,Wang,Urbas,Trudinger,Trudinger-Urbas,Trudinger-Wang}  for example. Herein, we  highlight some representative results from these contributions that  motivate our work.
	
	When discussing this field, it is essential to mention the boundary regularity results for the Monge-Amp\`ere equations. A complete proof of the regularity of the solution to the following problem
	\begin{align*}
	\det D^2u &= f(x)\quad \text{in} \; \varOmega \\
	u &= \varphi (x)\quad \text{on} \; \partial\varOmega
\end{align*}
can be traced back at least to 1977 in the work of Cheng and Yau \cite{Cheng-Yau}, 
 where $\varOmega$ is a bounded $C^2$ strictly convex domain, $\varphi\in C^2(\partial\varOmega)$, and 
 $f$ is a function of class $C^k$ ($k\geq3$) with the following behavior near the boundary:
\begin{equation*}
0<f(x)\leq A \operatorname{dist}(x, \partial \varOmega)^{\beta-n-1}
\end{equation*}
for some constants $A>0$ and $\beta>0$.  
This part laid the groundwork for later considerations that allow the right hand side  function $f$ to depend on $u$, but going over from the former to the latter is quite nontrivial. In the same paper  \cite{Cheng-Yau}, to facilitate the dependence of the right hand side on $u$,   	they  showed that  
	if  $\varOmega$ is a bounded $C^2$ strictly convex domain, $\varphi\in C^2(\partial\varOmega)$, and $f$ satisfies the following structural conditions:
	\begin{enumerate}[(i)]
		\item $f$ is a $C^k$ ($k\geq 3$) function defined on $\{(x,z): x\in\varOmega, t<u_0(\varphi,x)\}$, where \[u_0(\varphi,x)=\inf\{z: (x,z)\text{ belongs to the convex hull of the graph of $\varphi$ over $\partial\varOmega$}\};\]
		\item $f$ is non-decreasing in the second variable;
		\item for some constants $A>0$, $\alpha\geq 0$ and $\beta>0$, there holds
		\[0<f(x,z)\leq A\operatorname{dist}(x, \partial \varOmega)^{\beta-n-1}|z-u_0(\varphi,x)|^{-\alpha},\]
	\end{enumerate}
	then the solution $u$ to  
	the problem 
	\begin{align*}
		\det D^2u &= f(x,u)\quad \text{in} \; \varOmega \\
		u &= \varphi (x)\quad \text{on} \; \partial\varOmega
	\end{align*}
	exists,  belongs to $C^{k+1,\varepsilon}(\varOmega)$ for any $\varepsilon\in (0,1)$, and has the following boundary regularities: $u\in \operatorname{Lip}(\overline{\varOmega})$ if $\beta\geq n+1+\alpha$; $u\in C^\mu(\overline{\varOmega})$ for any exponent $\mu$ if $\beta\in [n+\alpha, n+1+\alpha)$; and $u\in C^{\frac{\beta}{n+\alpha}}(\overline{\varOmega})$  if $\beta\in (0,n+\alpha)$. 
	
	One of the earliest results regarding the boundary regularity for the Monge-Amp\`ere type equations, which permit the right hand side function $f$  to depend on $D u$, can be found in the work of Trudinger and Urbas \cite{Trudinger-Urbas} in 1983.   It was proved that the problem 
	\begin{align*}
		\det D^2u &= f(x,u,Du)\quad \text{in} \; \varOmega \\
		u &= \varphi (x)\quad \text{on} \; \partial\varOmega
	\end{align*}
	is uniquely solvable with convex solution 
	$u\in C^2(\varOmega)\cap C^{0,1}(\overline{\varOmega})$ under the assumption that $\varOmega$ is a $C^{1,1}$ uniformly convex domain in $\mathbb{R}^n$, $\varphi\in C^{1,1}(\overline{\varOmega})$, and $f$ is a positive function satisfying the following conditions:
	\begin{enumerate}[(i)]
		\item $f(x,z,q)\in C^{1,1}(\varOmega	\times\mathbb{R} \times \mathbb{R}^n)$,
		\item  $f_z\geq  0$ in $\varOmega	\times\mathbb{R} \times \mathbb{R}^n$,
		\item there exist a constant $N\geq 0$ and positive functions $g\in L^1({\varOmega})$, $h\in L^1_{loc}(\mathbb{R}^n)$ such that  $$\int_{\varOmega}g<\int_{\mathbb{R}^n}h, \ \  \ \  f(x,z,q)\leq  \frac{g(x)}{h(q)}\  \text{ for all } x\in
		\varOmega, |z|\geq  N, q\in \mathbb{R}^n,$$
		\item there exist a positive, nondecreasing function $A$ and constants $\beta\geq  n+1$,  $ 0\leq \gamma\leq  \beta$ such that
		$$0< f(x,z,q)\leq 
		A(|z|) \operatorname{dist}(x, \partial \varOmega)^{\beta-n-1}(1+|q|^2)^{\frac{\gamma}{2}}\ \text{ for all }(x,z,q)\in
		\varOmega	\times\mathbb{R} \times \mathbb{R}^n.$$
	\end{enumerate}
In 1988, 	Urbas \cite{Urbas} made a notable development along this direction by demonstrating that if $u\in C^2(\varOmega)$ is a convex solution 
	to 
	\[\det D^2u=f(x,u,Du),\]where $\varOmega$ is a $C^{1,1}$ bounded domain in $\mathbb{R}^n$, and $f$ is a positive function satisfying  that there exist some constants $A>0$, $\gamma> n+1$,  $ n+1\leq \beta<\gamma$ such that
	$$f(x,z,q)\geq  A\operatorname{dist}(x, \partial \varOmega)^{\beta-n-1}(1+|q|^2)^{\frac{\gamma}{2}}\ \text{ for all }
	(x,z,q)\in\varOmega	\times\mathbb{R} \times \mathbb{R}^n,$$
	then $u$ is globally H\"older continuous, and $|u|_{C^\mu(\overline{\varOmega})}\leq  C(\beta,\gamma,A,n,\varOmega)$, where $\mu=\min\big\{\frac{\gamma-\beta}{\gamma+n-2},\frac{1}{2}\big\}$. 
	
We note that these	boundary regularity results can be viewed as a type of boundary upper bound estimate.  It is  interesting to ask whether boundary lower bound estimate can also be obtained. Based on the work 
 \cite{Lazer-McKenna} in 1996, one can expect this to be the case, and moreover  both upper and lower bound estimates near the boundary appear promising to be established for the Dirichlet problems  with zero boundary condition.  Specifically, 
in \cite{Lazer-McKenna}, 	Lazer and McKenna  studied  different types of singular boundary value problems for the Monge-Amp\`ere operator by using the subsolution and supersolution method. 
	In particular, they proved that for the singular boundary problem
	\begin{align*}
		\det D^2u&=g(x)|u|^{-\alpha}\ \   
		\text{ in }\varOmega,\\
		u&=0\ \ \text{ on }\partial\varOmega,
	\end{align*}
	if $\varOmega$ is a smooth, bounded, strictly convex domain in $\mathbb{R}^n$,  $g\in C^\infty({\overline{\varOmega}})$ with $g(x)>0$ for all $x\in\varOmega$, and $\alpha>1$, then  there exist two positive constants $C_1$ and $C_2$ such that the solution $u$   
	can be bounded as follows:
	\begin{equation*}
		C_1\left(\operatorname{dist}(x,\partial\varOmega)\right)^{\frac{n+1}{n+\alpha}}\leq
		|u(x)|\leq C_2\left(\operatorname{dist}(x,\partial\varOmega)\right)^{\frac{n+1}{n+\alpha}}.
	\end{equation*} 
To the best of our knowledge, we have only found this related result so far, which has motivated us to explore boundary upper and lower bound estimates in the context of general fully nonlinear elliptic equations.  
	
For  general fully nonlinear elliptic equations, we now review some latest results that inspired this paper. In the past decade,	Cui, Jian and Li \cite{Cui-Jian-Li} provided the boundary H\"older estimates of strictly convex solutions in $C(\overline{\varOmega})\cap C^2(\varOmega)$ to the problem 
	\begin{align*}
		F(D^2u,Du,u,x) &= 0\quad \text{in} \; \varOmega \\
		u &= 0\quad \text{on} \; \partial\varOmega,
	\end{align*}
	where $\varOmega$ is a bounded convex domain satisfying the exterior sphere condition, and $F$ is an elliptic operator satisfying a series of bounded conditions. 
More recently, Li and Li \cite{Li-Li-N} employed the concept of $(a,\eta)$ type domains introduced by  Jian and Li \cite{Jian-Li} to further  discuss this problem, which can be rephrased as 
	\begin{align*}
		F(\lambda_{1}(D^{2}u),\cdot \cdot \cdot ,\lambda_{n}(D^{2}u)) &= f(x,u,Du)\quad \text{in} \; \varOmega,   \\
		u &= 0\quad \text{on} \; \partial\varOmega,
	\end{align*}
	and particularly focused on the relation between the boundary  H\"older regularity exponent $\mu$ and the convexity parameter $a$. Since boundary  H\"older regularity estimates can also be regarded as boundary upper bound estimates, they also established boundary lower bound estimates in \cite{Li-Li-F} for this problem with structure conditions given by inverse inequalities. We will not review more details here, as the  results in \cite{Li-Li-F,Li-Li-N} can be considered special cases of Theorem \ref{thm1} in this paper. The proofs in these works heavily rely on the subsolution and supersolution method, which is also the main technique that we will use in this study.

	The principal purpose of this paper is to establish boundary upper and lower bound estimates for more general Dirichlet problems with more general  boundary conditions. 
	Due to the  appearance of nonzero boundary conditions as included,  the problem \eqref{eq01}-\eqref{eq01boundary} is of independent interest and difficulties, and the existing boundary regularity theory  is still not complete and requires further exploration. 
	Based on the aforementioned literature,   estimating gradient terms and relating the boundary H\"older estimates to the convexity of boundary 
	are also key ingredients.

	\subsection{Structure conditions and main results}
	
	Motivated by the aforementioned lierature,  especially by  \cite{Li-Li-F,Li-Li-N}, we first propose the structure conditions for this paper.

	Suppose that for the boundary function $\varphi\in C(\partial\varOmega)$, there exist a convex function $\overline{\varphi}\in C(\overline{\varOmega})$ and a concave function $\underline{\varphi}\in C(\overline{\varOmega})$  such that 
	\begin{equation*}
		\overline{\varphi}\big|_{\partial\varOmega}=\varphi,\ \ \ \ \underline{\varphi}\big|_{\partial\varOmega}=\varphi.
	\end{equation*}
	Let	$\varphi^*$ be the following convex extension of $\varphi$ from $\partial\varOmega$ to $\overline{\varOmega}$:
	\begin{equation*}
		\varphi^*:=\sup\left\{\ell:\overline{\varOmega}\to\mathbb{R}:\ \ell\ \text{is an affine function},\  \ell\big|_{\partial\varOmega}\leq \varphi\right\}.
	\end{equation*}
	Let $\varphi_*$ be the following concave extension of $\varphi$ from $\partial\varOmega$ to $\overline{\varOmega}$: 
	\begin{equation*}
		\varphi_*:=\inf\left\{\ell:\overline{\varOmega}\to\mathbb{R}:\ \ell\ \text{is an affine function},\  \ell\big|_{\partial\varOmega}\geq \varphi\right\}.
	\end{equation*}
	Then it is clear that
	\begin{equation*}
		\varphi^*\big|_{\partial\varOmega}=\varphi,\ \ \ \ \varphi_*\big|_{\partial\varOmega}=\varphi.
	\end{equation*}

	For the right hand side of \eqref{eq01}, let $f:\varOmega \times (-\infty,0) \times \mathbb{R} ^{n}\rightarrow \mathbb{R}$ be a function satisfying the following structure conditions $(f_1)$-$(f_2)$ and $(f_3)$ or $(f_3')$: 
\begin{enumerate}
	\item[$(f_1)$] $f(x,z,q) \in C(\varOmega\times(-\infty,0)\times\mathbb{R}^{n})$;
	\item[$(f_2)$]   
	$f(x,z,q)$  is non-decreasing in the second argument for any fixed first and third arguments;
	\item[$(f_3)$] there exist some constants $A \geq 0$, $\alpha \in \mathbb{R}$, $\gamma \in \mathbb{R}$ and $\beta \geq n+1+\gamma$ such that
	\begin{equation*}
		0 \leq f(x, z, q) \leq A \operatorname{dist}(x, \partial \varOmega)^{\beta-n-1}|z-\varphi^*|^{-\alpha}\left(1+|q-D\varphi^*|^2\right)^{\frac{\gamma}{2}},  \forall(x, z, q) \in \varOmega \times(-\infty, 0) \times \mathbb{R}^n;
	\end{equation*}
	\item[$(f_3')$] there exist some constants $A \geq 0$, $\alpha \in \mathbb{R}$, $\gamma \in \mathbb{R}$ and 
	$\beta\geq n+1+\gamma$ such that
	\begin{equation*}
		f(x, z, q) \geq A \operatorname{dist}(x, \partial \varOmega)^{\beta-n-1}|z-\varphi_*|^{-\alpha}\left(1+|q-D\varphi_*|^2\right)^{\frac{\gamma}{2}},  \forall(x, z, q) \in \varOmega \times(-\infty, 0) \times \mathbb{R}^n.
	\end{equation*}
\end{enumerate}
It is evident from $(f_3)$ or $(f_3')$ that near the boundary $\partial\varOmega$, both the degenerate case $f(x,u,Du)\to0$ (when $\beta$ is large enough) and the singular case $f(x,u,Du)\to+\infty$ (when $\alpha$ is large enough) are contained in our discussion. Moreover, the conditions as $(f_3)$ and $(f_3')$ possess formal invariance under translations and rotations since the norms $|Du-D\varphi^*|$, $|Du-D\varphi_*|$ and all eigenvalues of $D^2u$ are invariant under these tranformations.

Let	$\lambda_{\min}(D^{2}u)$, $\lambda_{\max}(D^{2}u)$ denote respectively the  minimum and maximum of all $n$  eigenvalues of $D^2u$. 
	For the left hand side of \eqref{eq01}, let $F$ be an   operator satisfying the following structure conditions $(F_1)$ and $(F_2)$ or $(F_2')$: 
	\begin{enumerate}
		\item[$(F_1)$] $F$ is elliptic, i.e. \begin{equation*}
			\partial_{\lambda_{i}}F(\lambda_{1}(D^{2}u),\cdot \cdot \cdot ,\lambda_{n}(D^{2}u)) \geq 0,\ \forall i \in \{ 1,\cdot \cdot \cdot ,n \};
		\end{equation*}
		\item[$(F_2)$] there exist some constants $B>0$, $s\geq 0$ and $t\geq 0$ such that
		\begin{equation*}
			F(\lambda_{1}(D^{2}u),\cdot \cdot \cdot ,\lambda_{n}(D^{2}u))\geq B(\lambda_{\min}(D^{2}u))^{s}(\lambda_{\max}(D^{2}u))^{t};
		\end{equation*}
		\item[$(F_2')$] there some constants $B>0$, $s\geq 0$ and $t\geq 0$ such that
		\begin{equation*}
			F(\lambda_{1}(D^{2}u),\cdot \cdot \cdot ,\lambda_{n}(D^{2}u))\leq B(\lambda_{\min}(D^{2}u))^{s}(\lambda_{\max}(D^{2}u))^{t}.
		\end{equation*}
	\end{enumerate}
	We note that the following  class of $k$-Hessian operators for any $k\in\{1,\cdots,n\}$ satisfy these  conditions:
	\begin{align*}
		F(\lambda_{1}(D^{2}u),\cdot \cdot \cdot ,\lambda_{n}(D^{2}u))=\sigma_{k}(D^{2}u):=\sum_{1\leq i_1<\cdots<i_k\leq n} \lambda_{i_1}(D^{2}u) \cdots \lambda_{i_k}(D^{2}u),
	\end{align*}
	which include the Poisson operator $\Delta u$ when $k=1$ and the Monge-Amp\`ere operator $\det D^2 u$ when $k=n$ as two special cases. There are also many other operators that satisfy these conditions, 
	and we refer the readers to \cite{Li-Li-N} for more examples and details.

	Based on the above structure conditions, we are ready to state the main results of this paper. In order to analyze the relation of boundary estimates with the convexity of boundary, 
	these results are given for convex domain $\varOmega$ of exterior-$(a,\eta)$-type and any boundary point $x\in\partial\varOmega$ of exterior-$(a,\eta)$-type or  interior-$(a, \eta, \varepsilon)$-type  with its $\varOmega_{1/2,x}$ domain, see the definitions collected in Section \ref{sec:pre}. In this way, the boundary estimate exponent $\mu$ can be written as a function of the convexity parameter $a$:
	\begin{equation}\label{eq10}
		\mu(a):= \begin{cases}\frac{\beta-\gamma+2t-n-1}{\alpha-\gamma+s+t}+\frac{2 s}{a(\alpha-\gamma+s+t)}, & \text { if } \beta<\alpha+n+s-t+1-\frac{2 s}{a} \\ 
			1, & \text { if } \beta \geq \alpha+n+s-t+1-\frac{2 s}{a}\end{cases}
	\end{equation}
	where  $a=a(x)$ for pointwise boundary estimates  at any $x\in\partial\varOmega$, $a$ is taken as $a_{\max}:=\max_{x\in \partial\varOmega}a(x)$  for boundary upper bound estimates  at the entire boundary, and $a$ is taken as $a_{\min}:=\min_{x\in \partial\varOmega}a(x)$ for boundary lower bound estimates at the entire boundary.

	Our first two results concern the boundary upper bound estimates as well as the global H\"older regularity for  the general problem  
	(\ref{eq01})-\eqref{eq01boundary}.
	
	\begin{theorem}\label{thm1-1}
		Suppose that $ \varOmega \subseteq \mathbb{R} ^{n} $ 
		is a bounded convex domain, $\mu$ satisfies \eqref{eq10}, $\varphi\in C^{\mu}(\partial\varOmega)$,  $\varphi^* \in C^{\mu}(\overline{\varOmega})$, 
		$F$ satisfies $(F_1)$  and $(F_2)$, and $f$ satisfies $(f_1)$, $(f_2)$  and $(f_3)$. If $u$ is a convex viscosity solution to the  problem (\ref{eq01})-\eqref{eq01boundary},
		then for any boundary point $x\in \partial\varOmega$ that is of some exterior-$(a(x),\eta)$-type  with $a(x)\in[1,+\infty)$, 
		for any $y\in \varOmega_{1/2,x}$, there exists a constant $M=C(a, \eta, A,B, \alpha, \beta, \gamma, s, t, \operatorname{\operatorname{diam}}(\varOmega), n)>0$ such that 
		\begin{equation*}
			\vert u(y)-u(x) \vert\leq M\left(\operatorname{dist}(x,y)\right)^{\mu(a(x))},
		\end{equation*}
		and hence 
		\begin{equation*}
			\vert u(y)-u(x) \vert\leq M\left(\operatorname{dist}(x,y)\right)^{\mu(a_{\max})}.
		\end{equation*}
	\end{theorem}
	
		\begin{theorem}\label{thm1-1'}
		Suppose that $ \varOmega \subseteq \mathbb{R} ^{n} $ is exterior-$(a, \eta)$-type  domain with $a \subseteq[2,+\infty)$ (here corresponds to $a=a_{\max}$ in Theorem \ref{thm1-1}), $\mu$ satisfies \eqref{eq10}, $\varphi\in C^{\mu}(\partial\varOmega)$,  $\varphi^* \in C^{\mu}(\overline{\varOmega})$, 
		$F$ satisfies $(F_1)$  and $(F_2)$, and $f$ satisfies $(f_1)$, $(f_2)$  and $(f_3)$. If $u$ is a convex viscosity solution to the  problem (\ref{eq01})-\eqref{eq01boundary},
		then  for any 	$x,y\in\overline{\varOmega}$, there exists a constant $M=C(a, \eta, A,B, \alpha, \beta, \gamma, s, t, \operatorname{\operatorname{diam}}(\varOmega), n)>0$ such that 
		\begin{equation*}
			\vert u(y)-u(x) \vert\leq M\left(\operatorname{dist}(x,y)\right)^{\mu(a)},
		\end{equation*}
		and therefore $$u\in C^{\mu(a)}(\overline{\varOmega} )$$ with
		\begin{equation*}
			| u\vert_{C^{\mu(a)}(\overline{\varOmega})} \leq C(a, \eta, A,B, \alpha, \beta, \gamma, s, t, \operatorname{\operatorname{diam}}(\varOmega), n, |\varphi^*|_{C^{\mu(a)}(\overline{\varOmega})} ).
		\end{equation*}
	\end{theorem}
	
	Our third result concerns the boundary lower bound estimates for the general problem  
	(\ref{eq01})-\eqref{eq01boundary}.
	\begin{theorem}\label{thm1-2}
		Suppose that  $ \varOmega \subseteq \mathbb{R} ^{n} $ 
		is a bounded convex domain, $\mu$ satisfies \eqref{eq10},  $\varphi\in C^{\mu}(\partial\varOmega)$,  $\varphi_* \in C^{\mu}(\overline{\varOmega})$, 
		$F$ satisfies $(F_1)$  and $(F_2')$, and $f$ satisfies $(f_1)$, $(f_2)$ and $(f_3')$. If $u$ is a convex viscosity solution to problem (\ref{eq01})-\eqref{eq01boundary},
		then for any boundary point $x\in \partial\varOmega$ that is of some interior-$(a(x),\eta,\varepsilon)$-type  with $a(x)\in[1,+\infty)$, and for any $y\in \varOmega_{1/2,x}$, there exists a constant $m=C(a, b,\eta,\varepsilon, A,B,\alpha,\beta,\gamma,s,t, n)>0$ such that 
		\begin{equation*}
			\vert u(y)-u(x) \vert\geq m\left(\operatorname{dist}(x,y)\right)^{\mu(a(x))}.
		\end{equation*}
		and hence 
		\begin{equation*}
			\vert u(y)-u(x) \vert\geq m\left(\operatorname{dist}(x,y)\right)^{\mu(a_{\min})}.
		\end{equation*}
	\end{theorem}

		Our fourth result concerns the pointwise boundary estimates for the zero Dirichlet boundary problem, i.e. 
		the following particular case of the problem \eqref{eq01}-\eqref{eq01boundary} with 
	\begin{align}
		F(\lambda_{1}(D^{2}u),\cdot \cdot \cdot ,\lambda_{n}(D^{2}u))&=B(\lambda_{\min}(D^{2}u))^{s}(\lambda_{\max}(D^{2}u))^{t},\label{eq:Fcase}\\
		f(x,u,Du)&=A \operatorname{dist}(x, \partial \varOmega)^{\beta-n-1}|u|^{-\alpha}\left(1+|Du|^2\right)^{\frac{\gamma}{2}},\label{eq:fcase}\\
		\varphi(x)&\equiv0.\label{eq:varphicase}
	\end{align}

	\begin{theorem}\label{thm1}
		Suppose that  $ \varOmega \subseteq \mathbb{R} ^{n} $ 
		is a bounded convex domain,  $\mu$ satisfies \eqref{eq10}, $\varphi$ satisfies \eqref{eq:varphicase}, 
		$F$ satisfies $(F_1)$, $(F_2)$ and $(F_2')$, i.e. $F$ can be written as \eqref{eq:Fcase}, and 
		$f$ satisfies $(f_1)$, $(f_2)$, $(f_3)$ with $\varphi^*\equiv0$, and  $(f_3')$ with $\varphi_*\equiv0$, i.e. $f$ can be written as \eqref{eq:fcase}.  If $u$ is a convex viscosity solution to the problem (\ref{eq01})-\eqref{eq01boundary},
		then for any $x\in \partial\varOmega$ that is of some exterior-$(a(x),\eta)$-type and  interior-$(a(x),\eta',\varepsilon)$-type with $a(x)\in[1,+\infty)$,
		for any $y\in \varOmega_{1/2,x}$,  
		there exist constants 
		$M=C(a, \eta, A,B, \alpha, \beta, \gamma, s, t, \operatorname{\operatorname{diam}}(\varOmega), n)>0$ and $m=C(a, b,\eta',\varepsilon, A,B,\alpha,\beta,\gamma,s,t, n)>0$ such that  
		\begin{equation}\label{eq09}
			m\left(\operatorname{dist}(x,y)\right)^{\mu(a(x))} \leq \vert u(y)-u(x) \vert\leq M\left(\operatorname{dist}(x,y)\right)^{\mu(a(x))}.
		\end{equation}
		Moreover, if every boundary point $x\in\partial\varOmega$ is of some exterior-$(a(x),\eta)$-type with $a(x)\in[2,+\infty)$, 
		we have $$u\in C^{\mu(a_{\max})}(\overline{\varOmega} )$$ and 
		\begin{equation}\label{eq11}
			| u\vert_{C^{\mu(a_{\max})}(\overline{\varOmega} )} \leq C(a_{\max}, \eta, A,B, \alpha, \beta, \gamma, s, t, \operatorname{\operatorname{diam}}(\varOmega), n).
		\end{equation}
	\end{theorem}
	
	\begin{remark}
		In fact,  (\ref{eq09}) arises from a property of the solution, which states that the order of the solution to the problem will equal 
		$\mu(a(x))$ as it approaches the boundary point $x$. The proof of  (\ref{eq11}) relies on the second inequality of   (\ref{eq09}).  
	\end{remark}
	
	Once we have proved Theorem \ref{thm1-1}, Theorem \ref{thm1-1'} and Theorem \ref{thm1-2}, then Theorem \ref{thm1} follows immediately. To prove these results, we  will proceed as follows. Firstly, we investigate the properties of pointwise boundary estimates, 
	and   propose the concepts of upper-$(\mu,M)$-type and lower-$(\nu,m)$-type  (see Definition \ref{def2}) for this purpose. Next, we  demonstrate the validity of  global H\"older estimates  
	based on pointwise boundary estimates.  
	Finally, we  show that the solution to the problem \eqref{eq01}-\eqref{eq01boundary} satisfies  (\ref{eq09}) and determine that the global H\"older  regularity satisfies (\ref{eq10}). 
	
	In this paper, the case $a\in [2,+\infty)$ is the principal part, which is  
	dealt with in 
	Sections \ref{sec:sub} and \ref{sec:super}. We leave the case $a\in [1,2)$ and other discussions to  Section \ref{sec:discussion}. Precisely, 
	the rest of this paper is organized as follows. In Section \ref{sec:pre}, we present basic notations,  preliminary estimates and several assumptions to simplify the proof. Section  \ref{sec:sub} is devoted to the main proof of boundary upper bound estimates (i.e. Theorem \ref{thm1-1'}, Case $a\in [2,+\infty)$ in Theorem \ref{thm1-1} and the first part of Theorem \ref{thm1}) by constructing suitable subsolutions, while Section \ref{sec:super} is for the main proof of boundary lower bound estimates (i.e. Case $a\in [2,+\infty)$ in  Theorem \ref{thm1-2} and the second part of Theorem \ref{thm1}) by  constructing appropriate supersolutions.  Some further discussions and applications  of our results are given  in  Section \ref{sec:discussion}.  
	Finally, we provide some detailed computations on the barrier function $W$   in the appendix.

	\section{Preliminaries}\label{sec:pre}
	
	In this section, we propose some basic definitions and relevant lemmas for next sections.
	
	\subsection{Barrier function}  
	
	We consider functions as
	$W(x)=W(r, x_{n})$, where
	$x=\left(x^{\prime}, x_n\right), x^{\prime}=\left(x_1, \cdots, x_{n-1}\right)$ and $r=\left|x^{\prime}\right|=\sqrt{x_1^2+\cdots+x_{n-1}^2}$.
	Denote
	\begin{align*}
		&H[W]:=F(\lambda_1(D^2W),\cdots,\lambda_n(D^2W))\cdot [f(x,W,DW)]^{-1},\stepcounter{equation}\tag{\theequation}\label{eq37-HW}\\
		&\widetilde{H}[W]:=F(\lambda_1(D^2W),\cdots,\lambda_n(D^2W))\cdot [\widetilde{f}(x,W,DW)]^{-1},\stepcounter{equation}\tag{\theequation}\label{eq37}
	\end{align*}
where 
\begin{equation}\label{eq-widetildef}
	\widetilde{f}(x,W,DW):= A^{-1} \operatorname{dist}(x, \partial \varOmega)^{n+1-\beta}|W|^\alpha\left(1+|DW|^2\right)^{-\frac{\gamma}{2}}.
\end{equation}
Write for $i,j \in\{1,2, \cdots, n\}$ that $$ W_r:=\frac{\partial W}{\partial r}, \ \  W_i:=\frac{\partial W}{\partial x_i}, \ \   W_{ij}:=\frac{\partial^2 W}{\partial x_i\partial x_j}.$$ 
	
	We recall the following lemma from \cite{Li-Li-N,Li-Li-F} about  eigenvalues of the Hessian matrix $D^2 W$.
	
	\begin{lemma}\label{lemma:eigen}
		If the function $W$ satisfies
		$$
		W_{r r}>0, \quad W_{n n}>0, \quad W_{r r} \cdot W_{n n}-\left|W_{r n}\right|^2>0 \quad \text { in } \varOmega,
		$$
		then 	all the $n$ eigenvalues of the Hessian matrix $D^2 W$ are 
		\[\underbrace{\frac{W_r}{r}, \cdots, \frac{W_r}{r}}_{n-2\ \text{terms}},\underbrace{\frac{W_{r r}+W_{n n}-\sqrt{\left(W_{r r}-W_{n n}\right)^2+4\left|W_{r n}\right|^2}}{2}}_{:=\lambda_-(D^2 W)},\underbrace{\frac{W_{r r}+W_{n n}+\sqrt{\left(W_{r r}-W_{n n}\right)^2+4\left|W_{r n}\right|^2}}{2}}_{:=\lambda_+(D^2 W)},\]
		where $\lambda_-(D^2W)$ and $\lambda_+(D^2W)$ satisfy the following estimates:
		\begin{equation}\label{eq:eigen-relation}
			\lambda_-(D^2 W)\in\left[\frac{W_{r r} \cdot W_{n n}-\left|W_{r n}\right|^2}{W_{r r}+W_{n n}}, \min \left\{W_{r r}, W_{n n}\right\}\right],\ \ 
			\lambda_+(D^2 W)\in\big[ \max \left\{W_{r r}, W_{n n}\right\}, W_{r r}+W_{n n}\big]. 
		\end{equation}
		Moreover, 
		the trace and determinant of   $D^2 W$ can be written as 
		\begin{equation*}
			\operatorname{tr} D^2W=(n-2)\frac{W_r}{r}+W_{rr}+W_{nn},\ \ 
			\det D^2W=\left(\frac{W_r}{r}\right)^{n-2}\left(W_{r r} \cdot W_{n n}-\left|W_{r n}\right|^2\right),
		\end{equation*}
		and  the minimum and  maximum of all the $n$ eigenvalues of   $D^2 W$ can be bounded as
		\begin{equation}\label{eq:eigen-minmax}
			\lambda_{\min}(D^2 W) \in \left[\min \left\{\frac{W_r}{r}, \frac{W_{r r} \cdot W_{n n}-\left|W_{r n}\right|^2}{W_{r r}+W_{n n}}\right\},\frac{W_r}{r}\right], \ \ 	\lambda_{\max}(D^2 W) \in \left[ W_{n n} ,(n-2)\frac{W_r}{r}+W_{rr}+W_{nn}\right].
		\end{equation}
	\end{lemma}
	
	\begin{definition}\label{sufficient}
		A function $W$ is called a subsolution (supersolution) to the problem  \eqref{eq01}-\eqref{eq01boundary} over $\varOmega$ if it is  not only a subsolution (supersolution) to \eqref{eq01}  in $\varOmega$  but also satisfies 
		\[W\leq (\geq) u\ \text{ on }\partial\varOmega.\] Here, we say 
		$W$ is   a subsolution (supersolution) to  \eqref{eq01} if  
		\[F(\lambda_1(D^2W),\cdots,\lambda_n(D^2W))\geq (\leq)f(x,W,DW) \  \text{ in } \varOmega,\]
		i.e. 
		\[F(\lambda_1(D^2W),\cdots,\lambda_n(D^2W))-f(x,W,DW)\geq (\leq) 0 \  \text{ in } \varOmega,\]
		which is equivalent to  
		\[H[W]\geq (\leq) 1 \  \text{ in } \varOmega.\] 
	\end{definition}
	
	Following  \cite{Jian-Wang-Adv,Jian-Li,Li-Li-F,Li-Li-N}, we will construct suitable subdomains in $\varOmega$ and different auxiliary functions to be subsolution or supersolution in different domains. To end this subsection, we  prepare the barrier function required for these steps in the rest of this paper. 
	Let $b$ and $\xi$ be two positive constants (to be chosen), and consider the barrier function
	\begin{equation}\label{eq:W}
		W\left(r, x_n\right):=-\left(\left(\frac{x_n}{\xi}\right)^{\frac{2}{a}}-r^2\right)^{\frac{1}{b}}.
	\end{equation}
In next sections, we will construct subsolution $W_{\operatorname{sub}}$ (supersolution $W_{\operatorname{sup}}$) by using this barrier function with   $\xi$ sufficiently small (large). 
Here we list some direct and useful calculations for convenience: 
	\begin{align*}
		W_r&=\frac{2}{b}\cdot |W|^{1-b}\cdot r,\\
		W_n&=-\frac{2}{ab}\cdot |W|^{1-b}\cdot\big(\frac{x_n}{\xi}\big)^{\frac{2}{a}-1}\cdot\xi^{-1},\\
		W_{rr}&=\frac{2}{b}\cdot |W|^{1-b}-\frac{4(1-b)}{b^2}\cdot |W|^{1-2b}\cdot r^{2},\\
		W_{nn}&=-\frac{2(2-a)}{a^2b}\cdot |W|^{1-b}\cdot\big(\frac{x_n}{\xi}\big)^{\frac{2}{a}-2}\cdot\xi^{-2}-\frac{4(1-b)}{a^2b^2}\cdot |W|^{1-2b}\cdot\big(\frac{x_n}{\xi}\big)^{\frac{4}{a}-2}\cdot\xi^{-2},\\
		W_{rn}&=\frac{4(1-b)}{ab^2}\cdot |W|^{1-2b}\cdot r\cdot\big(\frac{x_n}{\xi}\big)^{\frac{2}{a}-1}\cdot\xi^{-1}.
	\end{align*}
	
	\subsection{Convex domain}
We first recall
	a careful description about the  convexity of domains from \cite{Jian-Li,Li-Li-N}.
	\begin{definition}\label{def1}
		Suppose that $\varOmega$ is a bounded convex domain in $\mathbb{R}^n$ and $P \in \partial \varOmega$. 
		\begin{enumerate}[(i)]
			\item We say $P$ is an exterior-$(a, \eta)$-type point (or say $\partial \varOmega$ satisfies exterior-$(a, \eta)$-type condition at $P$)  if there exist numbers $a \in[1,+\infty)$ and $\eta>0$ such that after translation and rotation transforms, we have
			$$
			P=O \quad \text { and } \quad \varOmega \subseteq\left\{x \in \mathbb{R}^n: x_n \geq \eta\left|x^{\prime}\right|^a,\  x_n\leq\operatorname{diam}(\varOmega),\  |x'|\leq\operatorname{diam}(\varOmega)\right\}.
			$$ 
			$\varOmega$ is called an exterior-$(a, \eta)$-type domain if its every boundary point is of exterior-$(a, \eta)$-type  with $a \in[2,+\infty)$.
			\item We say $P$ is an interior-$(a, \eta, \varepsilon)$-type point (or say $\partial \varOmega$ satisfies  interior-$(a, \eta, \varepsilon)$-type  condition at $P$) if there exist numbers $a \in[1,+\infty)$, $\eta>0$ and $\varepsilon>0$ such that after translation and rotation transforms, we have
			$$
			P=O \quad \text { and } \quad\left\{x \in \mathbb{R}^n:\left|x^{\prime}\right|<\mu,\  \tfrac{1}{2} \eta\left|x^{\prime}\right|^a < x_n < \tfrac{1}{2} \eta\varepsilon^a\right\} \subseteq \varOmega \subseteq \mathbb{R}_{+}^n.
			$$
		\end{enumerate}
	\end{definition}
	
		\begin{remark}\label{remark:diam}
		For any bounded convex domain, 
		by affine transforms, we may assume that $\operatorname{diam}(\varOmega)<1$.
	\end{remark}

	\begin{remark}\label{remark:etaepsilon}
		The concept of exterior-$(a,\eta)$-type   describes the convexity of boundary point and the convexity of entire boundary via external domains. 
		The smaller  the number $a$, the more convex  the entire boundary (and also the more convex  the  boundary locally). In contrast, the concept of interior-$(a, \eta, \varepsilon)$-type is a local concept, which only describes the convexity of boundary point via internal domains. When $a$ is fixed, provided that the point $(x'=\varepsilon,x_n=\frac{1}{2}\eta\varepsilon^a)$ lies  in $\varOmega$, it holds that the larger the number $\varepsilon$, the more convex the boundary locally.  Furthermore, a boundary point of interior-$(a, \eta, \varepsilon)$-type is also of interior-$(a,\eta',\varepsilon')$-type with some constants $\eta'>\eta$ and $\varepsilon'<\varepsilon$. It shows  that the parameters $\eta$ and $\varepsilon$ are in a generalized inverse proportion. Hence, for a fixed number $a$, without loss of generality, we assume that $\eta$ is sufficiently large,  $\varepsilon$ is sufficiently small, 
		and  $\eta\varepsilon^a$ equals to a positive constant value denoted by $C_{\eta,\varepsilon}$, i.e. 
		\begin{equation}\label{eq4.10}
			\eta\varepsilon^a \equiv   C_{\eta,\varepsilon},
		\end{equation}
		which will be frequently used when we  construct supersolutions in Section \ref{sec:super}. 
	\end{remark}

	\begin{remark}
		Once we have an exterior-$(a,\eta)$-type condition for a convex domain, then constructing an interior-$(a, \eta', \varepsilon)$-type condition can further help us subtly analyze (the upper bound of curvature and hence) the  regularity  of this domain in a local sense. The following Figure \ref{Fig.1} displays descriptions such that after suitable translations and rotations,   $\varOmega$ is an exterior-$(a,\eta)$-type domain and the origin is also an interior-$(a, \eta', \varepsilon)$-type point. To ensure the coherence and significance,   we will merely focus  on  estimates  throughout the entire boundary $\partial\varOmega$ for 	
		the case   $a\in[2,+\infty)$ in the main proof of this paper, and leave   discussions on more refined  estimates at each boundary point to Section \ref{sec:discussion}. Though there is no exterior-$(a,\eta)$-type domain with $a\in [1,2)$, some boundary points can be of exterior-$(a,\eta)$-type with $a\in [1,2)$ and the estimates at these points will also be briefly addressed in Section \ref{sec:discussion}. 
	\end{remark}

	\begin{figure}[htbp]  
		\centering 
		\includegraphics[scale=0.42]{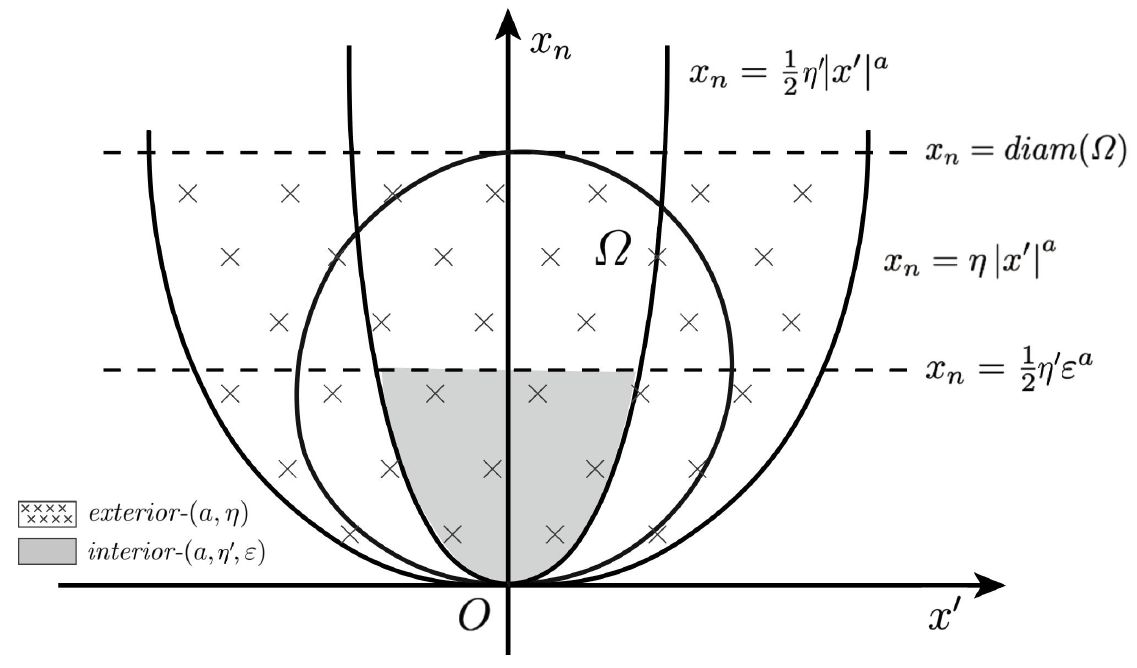} 
		\caption{Exterior-$(a,\eta)$-type domain $\varOmega$ and interior-$(a, \eta', \varepsilon)$-type point $O$. 
		} 
		\label{Fig.1} 
	\end{figure}
	
	\begin{remark}\label{remark:eta'}
		If $\varOmega$ is of exterior-$(a,\eta)$-type at $P$ for some $a\in[1,+\infty)$ and $\eta>0$, then $\varOmega$  is also of  interior-$(a,\eta',\varepsilon)$-type at $P$ for some $\eta'>0$ and $\varepsilon>0$. 
	In fact, for any boundary point, both $\sup\eta$ and $\inf\eta'$ (the best value for $\eta$ and $\eta'$ respectively) are given numbers. 
	\end{remark}
	
	Now we refer to the geometric handling in \cite{Li-Li-F} to bound $\operatorname{dist}(x, \partial \varOmega)$ via $x_n$. The result can be written as the following lemma.
\begin{lemma}[Local estimates of $\operatorname{dist}(x, \partial \varOmega)$]\label{lemma:local-dx}
	Suppose that  $ \varOmega \subseteq \mathbb{R} ^{n} $ 
	is a bounded convex domain and $P \in \partial \varOmega$ is an interior-$(a, \eta, \varepsilon)$-type point with $a \in[1,+\infty)$. Without loss of generality, we assume that
	$$
	P=0 \quad \text { and } \quad\left\{x \in \mathbb{R}^n:\left|x^{\prime}\right|<\mu, \frac{1}{2} \eta\left|x^{\prime}\right|^a<x_n<\frac{1}{2} \eta\varepsilon^a\right\} \subseteq \varOmega \subseteq \mathbb{R}_{+}^n.
	$$
	Consider the subdomain $V\subseteq \varOmega$ as 
	\begin{equation}\label{eq:V}
		V:=\left\{x \in \mathbb{R}^n:\left|x^{\prime}\right|<\big(\frac{1}{4}\big)^{\frac{1}{a}} \varepsilon,\  \eta\left|x^{\prime}\right|^a<x_n < \frac{1}{4} \eta\varepsilon^a\right\},
	\end{equation}
	and split its boundary as $\partial V=L\cup S$, where
	\begin{align*}
		& L:=\left\{x \in \mathbb{R}^n:\left|x^{\prime}\right| \leq\big(\frac{1}{4}\big)^{\frac{1}{a}} \varepsilon,\  x_n=\frac{1}{4} \eta\varepsilon^a\right\}, \\
		& S:=\left\{x \in \mathbb{R}^n:\left|x^{\prime}\right| \leq\big(\frac{1}{4}\big)^{\frac{1}{a}} \varepsilon,\  x_n=\eta\left|x^{\prime}\right|^a\right\}.
	\end{align*}
	Then for any $x\in V$, we have 
	\[\operatorname{dist}(x, \partial \varOmega)\in\left[2^{-1}\big(1+(a\eta \varepsilon^{a-1})^2\big)^{-\frac{1}{2}}\cdot x_n,x_n\right].\]
\end{lemma}

\begin{proof}
	For any $x\in V\subseteq\varOmega$,	it is obvious to see $\operatorname{dist}(x, \partial \varOmega)\leq x_n$, and hence our task remains to proving 
	\begin{equation}\label{eq:dx-lowbound}
		\operatorname{dist}(x, \partial \varOmega)\geq 2^{-1}\big(1+(a\eta \varepsilon^{a-1})^2\big)^{-\frac{1}{2}}\cdot x_n.
	\end{equation}
	To prove this, we construct a domain $\widetilde{V}$ such that $V\subseteq\widetilde{V}\subseteq \varOmega $ as  follows:
	\[\widetilde{V}:=\left\{x \in \mathbb{R}^n:\left|x^{\prime}\right|<\varepsilon, \frac{1}{2} \eta\left|x^{\prime}\right|^a<x_n<\frac{1}{2} \eta\varepsilon^a\right\}\]
	with its boundary $\partial\widetilde{V}=\widetilde{L}\cup\widetilde{S}$ as
	\begin{equation*}
		\begin{aligned}
			& \widetilde{L}:=\left\{x \in \mathbb{R}^n:\left|x^{\prime}\right| \leq \varepsilon, x_n=\frac{1}{2}\eta'\varepsilon^a \right\}, \\
			& \widetilde{S}:=\left\{x \in \mathbb{R}^n:\left|x^{\prime}\right| \leq \varepsilon, x_n=\frac{1}{2} \eta'\left|x^{\prime}\right|^a\right\} .
		\end{aligned}
	\end{equation*}
	
	For any $x\in V\subseteq \widetilde{V}$, there exists some point $z\in \partial \widetilde{V}$ such that 
	\begin{equation*} 
		\operatorname{dist}(x,\partial \varOmega) \geq \operatorname{dist}(x,\partial \widetilde{V})= \operatorname{dist}(x,z).
	\end{equation*}
	Let $l_1$ be a straight line passing through $x$ perpendicular to the $x_n$ axis and intersecting the boundary of $\widetilde{V}$ with $y_1$, while $l_2$  be a straight line passing through $x$ parallel to the $x_n$ axis and intersecting the boundary of $\widetilde{V}$ with $y_2$. It is easy to verify that $z\in \partial \widetilde{S}$ and
	\begin{equation*} 
		|x-z| \geq  \operatorname{dist}(x,\overline{y_{1}y_{2}} )=\frac{|x-y_1|\cdot|x-y_2|}{|y_1-y_2|} =\frac{1}{\sqrt{1+\left(\frac{\left|x-y_2\right|}{\left|x-y_1\right|}\right)^2}} \cdot\left|x-y_2\right|.
	\end{equation*}
	Then our target is reduced to the estimate for $|x-y_2|$ and $\frac{|x-y_2|}{|x-y_1|}$. Since $x\in V$, it is clear to  get  
	\begin{equation*}
		\left|x-y_2\right|=x_n-\frac{1}{2} \eta'\left|x^{\prime}\right|^a \geq x_n-\frac{1}{2} x_n=\frac{1}{2} x_n.
	\end{equation*}
	On the other hand, we can infer
	\begin{equation*}
		\begin{aligned}
			\frac{\left|x-y_2\right|}{\left|x-y_1\right|}&=\frac{x_n-\frac{1}{2} \eta'\left|x^{\prime}\right|^a}{\left(\frac{x_n}{\frac{1}{2} \eta'}\right)^{\frac{1}{a}}-\left|x^{\prime}\right|}=\frac{1}{2} \eta' \frac{\left(\left(\frac{x_n}{\frac{1}{2} \eta'}\right)^{\frac{1}{a}}\right)^a-\left|x^{\prime}\right|^a}{\left(\frac{x_n}{\frac{1}{2} \eta'}\right)^{\frac{1}{a}}-\left|x^{\prime}\right|}\leq \frac{1}{2} \eta' a\left(\left(\frac{x_n}{\frac{1}{2} \eta'}\right)^{\frac{1}{a}}\right)^{a-1}\\
			& \leq \frac{1}{2} \eta' a\left(\left(\frac{\frac{1}{4} \eta' \varepsilon^a}{\frac{1}{2} \eta'}\right)^{\frac{1}{a}}\right)^{a-1}  =2^{\frac{1}{a}-2} \eta' a \varepsilon^{a-1} \leq \eta' a \varepsilon^{a-1}.
		\end{aligned}
	\end{equation*}
	Combining the above estimates, 
	we obtain \eqref{eq:dx-lowbound} immediately. The lemma is proved. 
\end{proof}
	
	Let us recall the exterior and  interior sphere conditions, and investigate their relations with exterior-$(a,\eta)$-type and  interior-$(a, \eta, \varepsilon)$-type.
	
	\begin{definition}
		We say that a domain $\varOmega$ in $\mathbb{R}^n$ satisfies exterior (interior) sphere condition at some point  $x_0\in\partial\varOmega$  if  
		there exists an exterior (interior) ball $B_R(y_0)\supseteq\varOmega$ ($B_R(y_0)\subseteq \varOmega$) such that $\partial B_R(y_0)\cap\partial\varOmega=\{x_0\}$. 
	\end{definition}
	
	\begin{lemma}[See Lemma 2.1 in  \cite{Jian-Li}]\label{lemma:Jian-Li}
		For a bounded convex domain,   $(2,\eta)$ type  is equivalent to the    exterior sphere condition. Precisely, any $(2,\eta)$ type domain satisfies  exterior sphere condition with radius $R=\max\{\frac{1}{\eta},\operatorname{diam}(\varOmega)\}$; any bounded convex domain satisfying exterior sphere condition with radius $R$ is of $(2,\frac{1}{2R})$ type.
	\end{lemma}

	\begin{lemma}\label{lemma:interior}
		For a bounded convex domain $\varOmega$ and a boundary point $P\in\partial\varOmega$, we have the following two conclusions:
		\begin{enumerate}[(i)]
			\item a domain that is of  exterior-$(a,\eta)$-type   at $P$ with $a\in [1,2]$  satisfies the exterior sphere condition at $P$;
			\item a domain that is of  interior-$(a,\eta,\varepsilon)$-type at $P$ with $a\in [2,+\infty)$  satisfies the interior sphere condition at $P$.
		\end{enumerate}
	\end{lemma}
	\begin{proof}
		For (i), if   $\varOmega$ is of  exterior-$(a,\eta)$-type at $P\in\partial\varOmega$, then by translations and rotations, we can assume 
		\begin{equation}\label{eq:relation1}
		P=O\ \  \text{ and } \ \ \varOmega \subseteq\left\{x \in \mathbb{R}^n: x_n \geq \eta\left|x^{\prime}\right|^a,\  x_n\leq\operatorname{diam}(\varOmega),\  |x'|\leq\operatorname{diam}(\varOmega)\right\}.
		\end{equation} 
Hence, for any $x\in\varOmega$, by	taking $R=\max \{ \frac{1}{\eta},\operatorname{diam}(\varOmega) \}$, we obtain
		\begin{equation*}
			x_n\geq \eta |x'|^{a}\geq \eta |x'|^{2}\geq 	\frac{|x'|^2}{R}\geq \frac{|x'|^2}{R+\sqrt{R^2-|x'|^{2}}} =	R-\sqrt{R^2-|x'|^{2}},
		\end{equation*}
	where we have used the condition $a\leq 2$ and Remark \ref{remark:diam} in the second inequality.
This further implies 
		\begin{equation}\label{eq:relation2}
			\left\{x \in \mathbb{R}^n: x_n \geq \eta\left|x^{\prime}\right|^a,\  x_n\leq\operatorname{diam}(\varOmega),\  |x'|\leq\operatorname{diam}(\varOmega)\right\}\subseteq 	B_R(Re_n).
		\end{equation}
		Thus, by \eqref{eq:relation1} and \eqref{eq:relation2}, we obtain 
		\[\varOmega\subseteq B_R(Re_n).\]
		
		For (ii),  if $\varOmega$ is of   interior-$(a, \eta, \varepsilon)$-type  at $P\in\partial\varOmega$, then   by translations and rotations, we can assume that
		\begin{equation}\label{eq:relation3}
		P=O\ \  \text{ and } \ \ 	\varOmega \supseteq\left\{x \in \mathbb{R}^n:\left|x^{\prime}\right|<\mu,\  \tfrac{1}{2} \eta\left|x^{\prime}\right|^a < x_n < \tfrac{1}{2} \eta\varepsilon^a\right\}.
		\end{equation} 
		Now we take $R=\min\{\frac{1}{\eta},\frac{1}{4}\eta\varepsilon^{2}\}$. For any $x\in B_R(Re_n)$, we have $x_n\geq 	R-\sqrt{R^2-|x'|^{2}}$, and then we can infer that
		\begin{equation*}
			x_n\geq 	R-\sqrt{R^2-|x'|^{2}}=\frac{|x'|^2}{R+\sqrt{R^2-|x'|^{2}}}\geq\frac{|x'|^2}{2R}\geq  \frac{1}{2} \eta  |x'|^{2}\geq  \frac{1}{2} \eta  |x'|^{a},
		\end{equation*}
		where we have used the condition $a\geq 2$ and Remark \ref{remark:diam} in the last inequality. 
		This yields
		\begin{equation}\label{eq:relation4}
			B_R(Re_n)\subseteq \left\{x \in \mathbb{R}^n:\left|x^{\prime}\right|<\mu,\  \tfrac{1}{2} \eta\left|x^{\prime}\right|^a < x_n < \tfrac{1}{2} \eta\varepsilon^a\right\}.
		\end{equation}
		Finally, we conclude from \eqref{eq:relation3} and 
		\eqref{eq:relation4} that
		\[\varOmega\supseteq B_R(Re_n).\]
		
		The proof of this lemma is now complete.
	\end{proof}

	We are now in a position to study some local domains near boundary points, which will help us construct global estimates   in next sections.

	\begin{definition}\label{def-3}
		Let $\varOmega$ be a bounded convex domain, $u \in C(\overline{\varOmega} )$  be a convex function, and $P\in\partial\varOmega$.
		Then for  any point  $Q\in\partial\varOmega$ such that $Q\neq P$ and $Q$ is the  nearest point to $P$,  since  $l_{PQ}$ is bounded, there exists some point $y_Q\in l_{PQ}$ such that  $$u(y_{Q}) = \min_{ l_{PQ}}u.$$
		Now we define the following subdomain $\varOmega_{1/2,P}\subseteq\varOmega$  
		generated by the starting point $P$:
		\begin{equation*}
			\varOmega_{1/2,P} :=\left\{ x\in \varOmega :\  x = \lambda P + (1-\lambda)y_{Q},\ \text{ where }\lambda \in (0,1),\ Q\in\partial\varOmega \text{ with }Q\neq P    \right\}.
		\end{equation*}
		Here, the symbol $y_Q\in l_{PQ}$  represents that $y_Q$ is collinear with $P$ and $Q$, and  $\varOmega_{1/2,P}$ is ensured to be a domain after  suitable translations and rotations on coordinate system (see \eqref{assumption-5} for example).
	\end{definition}

	We  take $\varOmega\subseteq \mathbb{R} ^{2}$ as an example to illustrate this definition. 
	Figure \ref{Fig.2} provides the schematic diagram  of $\varOmega_{1/2,P}$  
	for a special case when the convex domain $\varOmega$ is a disk and the graph of convex function $u$ is a cone. 
	Here, the darker the color, the smaller the function value.

	\begin{figure}[htbp] 
		\centering
		\subfigure 
		{
			\includegraphics[scale=0.35]{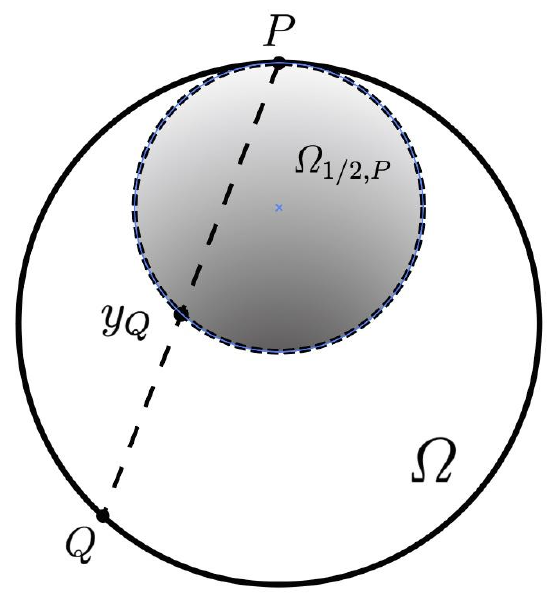} 
		}
		\hspace{20mm}
		\subfigure 
		{
			\includegraphics[scale=0.34]{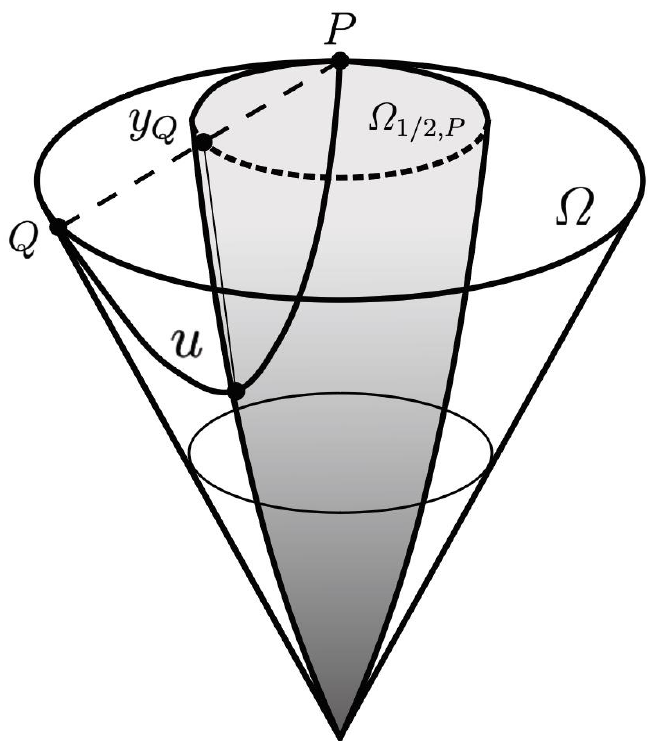}  
		}	
		\hspace{2mm}
		\caption{The schematic diagram of $\varOmega_{1/2,P}$. }
		\label{Fig.2}
	\end{figure}

	\begin{remark}\label{rmk}
		In  Definition \ref{def-3}, since  $u$ is a convex function and takes its minimum over $l_{PQ}$ at  $y_Q$, we point out that $u$ is monotonically decreasing   along $l_{Py_Q}$ (when $x$ moves from $P$ to $y_Q$) and  monotonically increasing along $l_{y_QQ}$ (when $x$ moves from $y_Q$ to $Q$).
	\end{remark}

	Let us list some properties of the domain $\varOmega_{1/2,P}$. The proofs will be omitted since they are easy to obtain.
	\begin{lemma}\label{property} $\varOmega_{1/2,P}$ satisfies the following properties:
		\begin{enumerate}[(i)]
			\item For any $x\in \varOmega_{1/2,P}$, we have $ u(x)\leq u(P)$.  
			\item If $\varOmega$ is of exterior-$(a,\eta)$-type at $P$, then  $\varOmega_{1/2,P} $ is also of exterior-$(a,\eta)$-type at $P$. 
			\item 
			If $\varOmega$ is an exterior-$(a,\eta)$-type domain, then  $\varOmega_{1/2,P} $ is also an exterior-$(a,\eta)$-type domain. 
			\item If $\varOmega$ is an exterior-$(a,\eta)$-type domain and  $\varOmega_{1/2,P}$ is an exterior-$(a',\eta')$-type domain, then we have $a' \leq  a$. 
			\item 
			If $\varOmega$ is of  interior-$(a, \eta, \varepsilon)$-type  at $P$ and  $\varOmega_{1/2,P}$ is of interior-$(a',\eta',\varepsilon')$-type at $P$, then we have $a' \leq  a$.
		\end{enumerate}
	\end{lemma}
	
	\begin{remark}
		It is not necessarily true that if $\varOmega$ is of  interior-$(a, \eta, \varepsilon)$-type  at $P$, then  $\varOmega_{1/2,P} $ is also of  interior-$(a, \eta, \varepsilon)$-type  at $P$. This can be valid by adding some condition  concerning the convex function over $\varOmega$, see  Proposition \ref{theA6} and Corollary \ref{coro} for example.  We will not discuss it further since this is not the focus of this paper.
	\end{remark}

\subsection{Convex function}
	
	We turn to consider convex functions over convex domains. 
	As a direct consequence of the monotonicity  estimate of convex functions, we have the following lemma.
	
	\begin{lemma}\label{lemma-2}
			Let $\varOmega$ be a bounded convex domain, $u \in C(\overline{\varOmega} )$  be a convex function, and $x\in\partial\varOmega$.
			For any two points $y_1,y_2\in\varOmega$, let $l_{y_1y_2}\cap \partial\varOmega=\{x_1, x_2\}$ and the four related  points be $x_1,y_1,y_2,x_2$ in order. It holds that
		\begin{equation*} 
			|u(y_2)-u(y_1)| \leq \mathop{\max}\{ |u(x_{1}+(y_2-y_1))-u(x_{1})|,|u(x_{2}-(y_2-y_1))-u(x_{2})|\}.
		\end{equation*}
	\end{lemma}
	
	Hence we get an upper bound concerning the deviation of values of convex functions at  different points in convex domains via near boundary estimates. It is time to further analyze this upper bound   
	and use the boundary linear approximation of convex functions over convex domains to give another classification for convex domains.

	\begin{definition}\label{def2}
		Let $\varOmega$ be a bounded convex domain, $x\in\partial\varOmega$, and $u$ be a convex function over $\varOmega$. 
		\begin{enumerate}[(i)]
			\item We say  $x$ is an  upper-$(\mu,M)$-type point about $u$ if for any  $y\in\varOmega_{1/2,x}$, there exist constants  $\mu\in(0,1]$ and $M>0$ such that 
			\begin{equation}\label{eq12}
				|u(y)-u(x)|\leq M\left(\operatorname{dist}(y,x)\right)^{\mu}.
			\end{equation}
			$\varOmega$ is called an upper-$(\mu,M)$-type domain about $u$ if its every boundary point is of upper-$(\mu,M)$-type. 
			\item We say   $x$ is a lower-$(\nu,m)$-type point about $u$ if for any  $y\in\varOmega_{1/2,x}$, there exist constants  $\nu\in(0,1]$ and $m>0$ such that 
			\begin{equation}\label{eq12-2}
				|u(y)-u(x)|\geq m\left(\operatorname{dist}(y,x)\right)^{\nu}.
			\end{equation}
			$\varOmega$ is called a lower-$(\nu,m)$-type domain about $u$ if its every boundary point is of lower-$(\nu,m)$-type. 
		\end{enumerate}
	\end{definition}

	\begin{remark}
		This definition gives  careful descriptions on the  convexity of functions over convex domains. 
		The equations \eqref{eq12} \eqref{eq12-2} mean that  
		the difference between the function's value at any point near boundary  
		and its value at the boundary has  upper and lower bound estimates. As depicted in Figure \ref{Fig.3}, these bounds depend on the distance from the point to the boundary and 
		specify the  order of this dependence.
	\end{remark}
	
		\begin{figure}[htbp]  
		\centering 
		\includegraphics[scale=0.49]{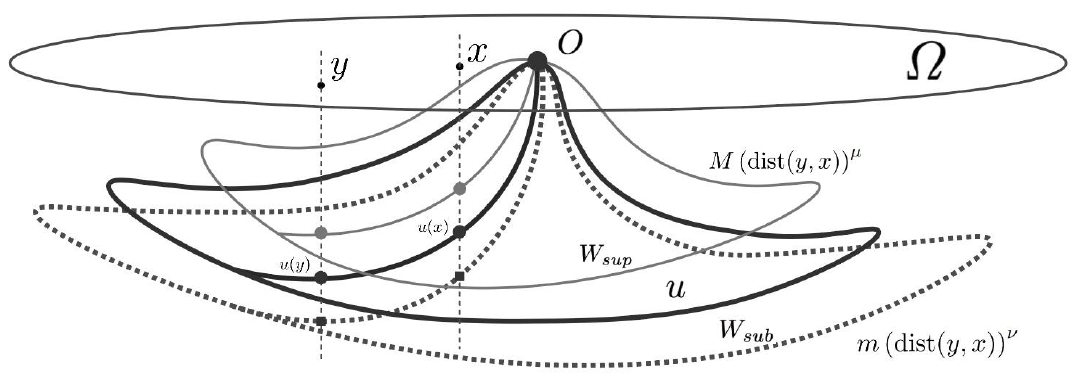} 
		\caption{Boundary linear approximation.} 
		\label{Fig.3} 
	\end{figure}

	\begin{remark}
	If $\varOmega$ is an exterior-$(a,\eta)$-type domain with $a\in [2,+\infty)$, then the following two distances $$\operatorname{dist}(y,\partial\varOmega) \sim  \operatorname{dist}(y,x)$$ are of the same order when $x\in \partial\varOmega$. As a result, in Definition \ref{def2} for exterior-$(a,\eta)$-type domains with $a\in [2,+\infty)$, the term $\operatorname{dist}(y,x)$ in \eqref{eq12} \eqref{eq12-2} 
	can be changed to that $\operatorname{dist}(y,\partial\varOmega)$.
\end{remark}
	
	\begin{remark}
		In Definition \ref{def2} (i), 
		the condition $y\in\varOmega_{1/2,x}$ can be changed to that $y\in\varOmega$. The resulting new definition of upper-$(\mu,M)$-type is equivalent to Definition  \ref{def2} (i) based on Definition \ref{def-3} and the fact that $u$ is convex over $\varOmega$. However, we cannot make this revision in Definition \ref{def2} (ii). 
	\end{remark}

\begin{remark}
	Let $\varOmega$ be a bounded convex domain, $u \in C(\overline{\varOmega} )$  be a convex function, and $x\in\partial\varOmega$. Assume there is a point $x_0\in\varOmega$ such that 
	\[u(x_0)=\min_{\varOmega}u.\]
\begin{enumerate}[(i)]
\item If	for   any $y\in l_{x_0x}$, there exist constants  $\mu\in(0,1]$ and $M>0$ such that 
\[|u(y)-u(x)|\leq M\left(\operatorname{dist}(y,x)\right)^{\mu},\]
then	it holds for any $y\in\varOmega_{1/2,x}$ that 
\[|u(y)-u(x)|\leq M\left(\operatorname{dist}(y,x)\right)^{\mu},\]
i.e. $\varOmega$ is an upper-$(\mu,M)$-type domain. 
\item If	for   any $y\in l_{x_0x}$, there exist constants  $\nu\in(0,1]$ and $m>0$ such that 
\[|u(y)-u(x)|\geq m\left(\operatorname{dist}(y,x)\right)^{\nu},\]
then	it holds for any $y\in\varOmega_{1/2,x}$ that 
\[|u(y)-u(x)|\geq m\left(\operatorname{dist}(y,x)\right)^{\nu},\]
i.e. $\varOmega$ is a lower-$(\nu,m)$-type domain. 
\end{enumerate}
\end{remark}

	\begin{remark}\label{remark:domain}
		To show that  $\varOmega$ is of upper-$(\mu,M)$-type, it suffices to verify \eqref{eq12} for any $x\in\partial\varOmega$ and $y\in\varOmega_{1/2,x}$. This can be transformed into examining the following estimate when we fix $x\in\partial\varOmega$ and $y\in\varOmega_{1/2,x}$:
		\begin{equation}\label{eq12'}
			\vert u(x+(y_{2}-y_{1})) - u(x)\vert \leq M\left(\operatorname{dist}(x+(y_{2}-y_{1}),x)\right)^{\mu},\ \ \ \  \forall y_1,y_2\in l_{xy},
		\end{equation}	
		which intuitively means that for any infinite sequence of points approaching the boundary via linear approximation (along some straight line), the diﬀerence between the function’s value at any
		point in the sequence and its value at the boundary has an upper bound estimate.  
		On the other hand, if  $\varOmega$ is of upper-$(\mu,M)$-type, we can also employ \eqref{eq12'} to match with Lemma \ref{lemma-2} to further imply interior upper bound estimates, which will be immediately used in the proof of Lemma \ref{lemma-new}. 
	\end{remark}

	Based on Lemma \ref{lemma-2} and Definition \ref{def2}, we are now ready to introduce the following global H\"older regularity result for convex functions with nonzero boundary condition.

	\begin{lemma}\label{lemma-new}
		Let $\varOmega$ be a bounded convex domain and $u \in C(\overline{\varOmega} )$  be a convex function. Suppose that  
		$u| _{\partial\varOmega} = \varphi$, where  $\varphi \in C(\partial\varOmega)$. If there exist constants $\mu\in(0,1]$ and $M>0$ such that 
		\[\text{$\varOmega$ is an upper-$(\mu,M)$-type domain about  $u$,}\]
		then we have
		$$ u \in C^{\mu}(\overline{\varOmega}),$$
		and
		\begin{equation}\label{eq:C-alpha}
			\Vert u\Vert _{C^{\mu}(\overline{\varOmega})} \leq M\left(\left(\operatorname{diam}(\varOmega)\right)^{\mu}+1\right)+\mathop{\max}_{\partial\varOmega} |\varphi|.
		\end{equation}
	\end{lemma}
	
	\begin{proof}
		For any two points $y_1,y_2\in\varOmega$, let $l_{y_1y_2}\cap \partial\varOmega=\{x_1, x_2\}$ and the four related  points be $x_1,y_1,y_2,x_2$ in order.  Without loss of generality, we can assume that 
		\begin{equation*}
			\mathop{\max}_{x\in l_{y_{1}y_{2}}\cap  \partial\varOmega} {| u(x+(y_{2}-y_{1})) - u(x)\vert } = | u(x_{1}+(y_{2}-y_{1})) - u(x_{1})|.
		\end{equation*}
		By Lemma \ref{lemma-2}, Definition \ref{def2} and Remark \ref{remark:domain}, it then follows that 
		\begin{align*}
			|u(y_2)-u(y_1)| &\leq    |u(x_{1}+(y_2-y_1))-u(x_{1})| \\ 
			&\leq M\left(\operatorname{dist}(x_1+(y_{2}-y_{1}),x_1)\right)^{\mu}\\
			&=M\left(\operatorname{dist}(y_2,y_1)\right)^{\mu}=M|y_2-y_1|^\mu\\
			&\leq M\left(\operatorname{diam}(\varOmega)\right)^{\mu}.
		\end{align*}
		From these relations, we can infer that
		\begin{align*}
			\Vert u\Vert _{C(\overline{\varOmega})}
			&=\mathop{\sup}_{\varOmega}\vert  u\vert 
			=	\mathop{\max}_{y_1,y_2\in\varOmega\atop x_1\in\partial\varOmega}\vert  u(x_{1}+(y_2-y_1))\vert \\
			&\leq	\mathop{\max}_{y_1,y_2\in\varOmega\atop x_1\in\partial\varOmega}\vert  u(x_{1}+(y_2-y_1))-u(x_1)\vert +\mathop{\max}_{ x_1\in\partial\varOmega}\vert u(x_1)\vert\\
			&\leq M\left(\operatorname{diam}(\varOmega)\right)^{\mu}+\mathop{\max}_{\partial \varOmega}| \varphi\vert,
		\end{align*}
		and
		\begin{align*}
			[u]_{C^{0,\mu}(\overline{\varOmega})} 
			&=  \mathop{\sup}_{y_1,y_2\in\varOmega\atop y_1\neq y_2} \left\{  \frac{|u(y_2)-u(y_1)\vert}{| y_2-y_1\vert^{\mu} } \right\}\leq M.
		\end{align*}
		Finally, adding up the last two inequalities yields the desired estimate \eqref{eq:C-alpha}. Thus the lemma is proved. 
	\end{proof}

	\subsection{Several simplifications}\label{sec:simply}
We collect all the assumptions that we will use to simplify the proof.

	The first assumption is about the relation of parameters. We let
	\begin{equation*}
		\beta <\alpha +s - t + n  +1 - \frac{2s}{a}. 
	\end{equation*}
	In fact, if $\beta \geq \alpha +s - t + n  +1 - \frac{2s}{a}$, then we can always find $\widetilde{\beta} <\alpha +s - t + n  +1 - \frac{2s}{a}$ to replace  $\beta$ and come to the conclusion that 
	$\frac{\widetilde{\beta}-\gamma+2 t-n-1}{\alpha-\gamma+s+t}+\frac{2 s}{a(\alpha-\gamma+s+t)}$ can be arbitrarily chosen in $(0,1]$ (here we can take $\mu=1$). Hence in what follows, we let $\mu$ take the value:
	\begin{equation}\label{assumption-1'}
		\mu(a) = \frac{\beta-\gamma+2t-n-1}{\alpha-\gamma+s+t}+\frac{2 s}{a(\alpha-\gamma+s+t)}.
	\end{equation}

	The others consist of a series of geometric assumptions on convex domains. Globally, due to  
	Definition \ref{def2}, we just need to discuss whether 
	\begin{equation*}
		\text{$\varOmega$ is of  upper-$(\mu,M)$-type at some boundary point}
	\end{equation*} 
	when proving the boundary upper bound estimates, 
	and whether 
	\begin{equation*}
		\text{$\varOmega$ is of  lower-$(\nu,m)$-type at some boundary point}
	\end{equation*} 
	when proving the boundary lower bound estimates. 
	Locally,  let us focus on the $\varOmega_{1/2,P}$ domain, where $P\in\partial\varOmega$. Since the problem \eqref{eq01} is invariant under translations and rotations, without loss of generality, we  assume that 
	\begin{equation}\label{assumption-5}
		\begin{split}
			&	\text{$P=O$},  \\ &\text{$x_{n}$-axis is in the same direction as the inward normal vector of $\varOmega_{1/2,P}$ at $P$},\\ 
			&\text{and $x'$-axis is   parallel to $Du(P)$.}
		\end{split}	
	\end{equation}   
	In particular, when the domain $\varOmega$ is of exterior-$(a,\eta)$-type at $P$, 
	it follows from Lemma \ref{property} (iii) that  $\varOmega_{1/2,P} $ is also
	of exterior-$(a,\eta)$-type at $P$, i.e. 
	\begin{equation*}
		\varOmega_{1/2,P} \subseteq\left\{x \in \mathbb{R}^n: x_n \geq \eta\left|x^{\prime}\right|^a\right\}.
	\end{equation*} 
	Now we choose $\delta \in(0,1)$ and $0<\xi=\xi(\delta, a, \eta)<1$ such that $\xi\left(\frac{1}{\delta}\right)^{\frac{a}{2}} \leq \eta$. Finally, we obtain the following relation among domains:
	\begin{equation}\label{assumption-7}
		\varOmega_{1/2,P}\subseteq\varOmega \subseteq\left\{\left.x \in \mathbb{R}^n\left|x_n \geq \eta\right| x^{\prime}\right|^a\right\} \subseteq\left\{x \in \mathbb{R}^n \left\lvert\, \delta\left(\frac{x_n}{\xi}\right)^{\frac{2}{a}} \geq r^2\right.\right\}.
	\end{equation} 
	We remark here that \eqref{assumption-7} is not required in Section \ref{sec:super} and indeed we only need to discuss the domain $V$ in Section \ref{sec:super}, see Remark \ref{remark:V}.

	The last assumption concerns 
	the boundary function.  After suitable translations if necessary, we suppose 
	\begin{equation}\label{assumption-3}
		\varphi(P)=0.
	\end{equation}
	The assumption \eqref{assumption-3} together with Lemma \ref{property} (ii) 
	further yields
	\begin{equation}\label{eq:assumption}
		u\big|_{\varOmega_{1/2,P}}\leq u(P)= \varphi(P)=0.
	\end{equation}

	\section{Boundary upper bound estimate}\label{sec:sub}
	
	This section is dedicated to proving the right hand side of the inequality (\ref{eq09}) in the case $a\in[2,+\infty)$, where we need boundedness conditions $(F_2)$ and $(f_3)$.

	\subsection{Construction of subsolution}
	By Lemma \ref{lemma-2}, we have constructed the upper bound for deviation of convex functions at different points via the estimates  near the boundary $\partial\varOmega$. However, it is difficult for us to give the corresponding estimates  away from the boundary $\partial\varOmega$. Hence we are not capable of   constructing a subsolution to the problem \eqref{eq01}-\eqref{eq01boundary}  
	in the whole domain $\varOmega$ (as this cannot be done away from  $\partial\varOmega$). 
	We note that a generalized convex solution 
	$u$ to \eqref{eq01}-\eqref{eq01boundary} in $\varOmega$ is also a generalized convex solution to \eqref{eq01}-\eqref{eq01boundary} in    the subdomain $\varOmega_{1/2,P}\subseteq\varOmega$ near  $\partial\varOmega$. Now we turn to choose suitable constants $b>0$ and $0<\xi<1$ such that 
	\begin{equation*}
	W_{\operatorname{sub}}:=W+\varphi^*
	\end{equation*}
(with 
	$W$ in the form of \eqref{eq:W}) is a subsolution   to the problem \eqref{eq01}-\eqref{eq01boundary} over $\varOmega_{1/2,P}$ or its suitable subdomain.

	We first claim that we can always take $\xi$  small enough to obtain
	\begin{equation*} 
		W \leq u \ \ \text{ on }\partial\varOmega_{1/2,P}.
	\end{equation*} 
	Indeed, if $\xi$ is sufficiently small, then there exists some neighborhood of $P$, denoted by $N(P)$,  such that the following holds outside this neighborhood:
	\begin{equation*}
		W \leq u \ \ \text{ on }\partial\varOmega_{1/2,P} \cap N(P)^c;
	\end{equation*}
	while in this neighborhood, by choosing some suitable constant $b$,  we first make the order of  $W$  the same as that of $u$, and then by choosing some sufficiently small constant $\xi$, we also get 
	\begin{equation*}
		W \leq u \ \ \text{ on }\partial\varOmega_{1/2,P} \cap N(P).
	\end{equation*}
	Combining these two cases yields the claim. 
By \eqref{eq:assumption} and the fact that $\varphi^*$ is convex, we always have $\varphi^*\leq 0$ and then $W_{\operatorname{sub}}\leq W$ over $\varOmega$.   Hence 
		\begin{equation}\label{eq31}
		W_{\operatorname{sub}} \leq u \ \ \text{ on }\partial\varOmega_{1/2,P}.
	\end{equation} 
It is clear that the above proof can be also used to prove $W_{\operatorname{sub}} \leq u$ on  the boundary of subdomains of $\varOmega_{1/2,P}$.

	Now we discuss how to choose the parameters $b$ and $\xi$ to make $W_{\operatorname{sub}}$ as a subsolution to  \eqref{eq01}, that is, to satisfy
		\begin{equation}\label{eq:HW}
		H[W_{\operatorname{sub}}] \geq 1. 
	\end{equation}
Based on the condition $(f_3)$ and the fact that $\varphi^*$ is convex, we note that 
\begin{align*}
H[W_{\operatorname{sub}}]&=F(\lambda_1(D^2W_{\operatorname{sub}}),\cdots,\lambda_n(D^2W_{\operatorname{sub}}))\cdot [f(x,W_{\operatorname{sub}},DW_{\operatorname{sub}})]^{-1}\\
&=F(\lambda_1(D^2W+D^2\varphi^*),\cdots,\lambda_n(D^2W+D^2\varphi^*))\cdot [f(x,W+\varphi^*,DW+D\varphi^*)]^{-1}\\
&\geq F(\lambda_1(D^2W),\cdots,\lambda_n(D^2W))\cdot \left[A^{-1} \operatorname{dist}(x, \partial \varOmega)^{n+1-\beta}|W|^\alpha\left(1+|DW|^2\right)^{-\frac{\gamma}{2}}\right]^{-1} \\
&=F(\lambda_1(D^2W),\cdots,\lambda_n(D^2W))\cdot[\widetilde{f}(x,W,DW)]^{-1}=\widetilde{H}[W].
\end{align*}
Hence we only need to ensure that
	\begin{equation}\label{eq:tildeHW}
		\widetilde{H}[W] \geq 1. 
	\end{equation}

	This part mainly refers to   the work \cite{Li-Li-N}, where the equation \eqref{eq01} is equipped with  zero  boundary condition $u= 0$ on $\partial\varOmega$, and $W_{\operatorname{sub}}$ has been constructed to satisfy \eqref{eq:HW} in $\varOmega$ and hence to be a subsolution  in $\varOmega$. 
	Here we  briefly recall the main steps of this construction:
	\begin{itemize}
		\item \textit{Step 1}: We consider $\frac{2(\alpha-\gamma+t)}{\beta-\gamma+2t-n-1}>2$, i.e. $\beta<\alpha-t+n+1$, and distinguish two substeps:
		\begin{itemize}
			\item \textit{Step 1-1}: Let $2 \leq a<\frac{2(\alpha-\gamma+t)}{\beta-\gamma+2t-n-1}$. In this case, we will find $b>1$ and $0<\xi<1$ such that \eqref{eq:HW} holds in $\varOmega_{1/2,P}$  or its suitable subdomain.
			\item \textit{Step 1-2}: Let $\frac{2(\alpha-\gamma+t)}{\beta-\gamma+2t-n-1} \leq a<+\infty$. In such a case, we will find $0<b\leq 1$ and $0<\xi<1$ such that \eqref{eq:HW} holds in $\varOmega_{1/2,P}$ or its suitable subdomain.
		\end{itemize}
		\item \textit{Step 2}: We consider $\frac{2(\alpha-\gamma+t)}{\beta-\gamma+2t-n-1}\leq 2$, i.e. $\beta\geq\alpha-t+n+1$. This case can be treated in the same manner as  in \textit{Step 1-2} since we always have  $\frac{2(\alpha-\gamma+t)}{\beta-\gamma+2t-n-1}\leq 2\leq a<+\infty$.
	\end{itemize}
	Therefore, it suffices to deal with the \textit{Step 1}. 
	The following procedures are required  in both  situations of \textit{Step 1}, and all constants therein are collected in Table \ref{table:constant}.

	\begin{table}[ht] \caption{Collection of constants in \textit{Step 1} of lower bound estimate for the case $a\in [2,+\infty)$}
		\label{table:constant}
		\renewcommand{\arraystretch}{1.8}\centering
		\begin{tabular}{|c|c|c|c|c|}
			\hline
			Symbol & Values in \textit{Step 1-1}
			& Values in \textit{Step 1-2}
			& Sign
			\\ [0.5ex]
			\hline
			$C_1$& $\frac{8(b-1)}{a^2 b^3} (1-\delta)^{a-2}$ & $\frac{8(b-1)}{a^2 b^3}(\delta(a-2)+1) +\frac{4(a-2)}{a^2 b^2}(1-\delta)^{a-1}$ &$C_1>0$ \\\hline
			$C_2$& $\left(\frac{2}{b}+\frac{4(b-1)}{b^2} \delta(1-\delta)^{-1}\right) \left(\frac{2(a-2)}{a^2 b}+\frac{4(b-1)}{a^2 b^2}\right)$ & $ \frac{4(a-2)}{a^2 b^2}+\frac{8(b-1)}{a^2 b^3}(1-\delta)^{a-2} $&$C_2>0$ \\\hline
			$C_3$& \begin{tiny}$
			\left(\frac{2}{b}+\frac{4(b-1)}{b^2} \delta(1-\delta)^{-1}\right) (\operatorname{\operatorname{diam}}(\varOmega))^{2-\frac{2}{a}}+\frac{2(a-2)}{a^2 b}+\frac{4(b-1)}{a^2 b^2}$\end{tiny} &\begin{tiny} $\frac{2}{b} (\operatorname{\operatorname{diam}}(\varOmega))^{2-\frac{2}{a}}+\frac{2(a-2)}{a^2 b}+\frac{4(b-1)}{a^2 b^2}(1-\delta)^{a-2}$\end{tiny} &$C_3>0$\\\hline
			$C_4$& \multicolumn{2}{|c|}{$\min \left\{\frac{2}{b}, \frac{C_1(a, b, \delta)}{C_3(a, b, \delta, \operatorname{\operatorname{diam}}(\varOmega))}\right\}$}  &$C_4>0$ \\\hline
			$C_5$& $  \frac{2(a-2)}{a^2 b}(1-\delta)^{a-1}+  \frac{4(b-1)}{a^2 b^2}(1-\delta)^{a-2}$ & $  \frac{2(a-2)}{a^2 b}(1-\delta)^{a-1}+ \frac{4(b-1)}{a^2 b^2}$   &$C_5>0$  \\\hline
			$C_6$& \multicolumn{2}{|c|}{$1+a^2   
				(\operatorname{\operatorname{diam}}(\varOmega))^{4-\frac{4}{a}}+\frac{a^2 b^2}{4}(1-\delta)^{2-a}    
				(\operatorname{\operatorname{diam}}(\varOmega))^{2-\frac{4}{a b}}$} &$ 
			C_6>0$\\
				\hline
			$C_7$& \multicolumn{2}{|c|}{$\left(\frac{ab}{2}\right)^{\gamma}\cdot\min \left\{1,C_6(a,b,\delta, \operatorname{\operatorname{diam}}(\varOmega))^{-\frac{\gamma}{2}}\right\}   \cdot \min \left\{1,(1-\delta)^{\left(1-\frac{a}{2}\right) \gamma}\right\}$  } &$ 
			C_7>0$   \\\hline
				$C_8$& \multicolumn{2}{|c|}{$\min\{2^{\beta-n-1}\big(1+(a\eta' \varepsilon^{a-1})^2\big)^{\frac{\beta-n-1}{2}},1\}$  } &$ 
			C_8>0$   \\
			\hline
			$C_9$&\multicolumn{2}{|c|}{ $A^{-1}B  C_4(a, b, \delta, \operatorname{\operatorname{diam}}(\varOmega))^s C_5(a, b, \delta)^t  C_7(a, b, \gamma, \delta, \operatorname{\operatorname{diam}}(\varOmega)
			)C_8(a,\eta',\varepsilon,\beta,n)$ }   &$ 
			C_9>0$  \\\hline
			$C_{10}$&\multicolumn{2}{|c|}{ $C_9(a, b, \eta',\varepsilon,A,B,\beta,\gamma,s,t, \delta,     \operatorname{\operatorname{diam}}(\varOmega), n)\cdot \min\left\{(1-\delta)^{\frac{\alpha}{b}+(\frac{a}{2}-\frac{1}{b})\gamma+(\frac{1}{b}-1)s+(\frac{1}{b}-a)t},1
				\right\}$ }   &$ 
			C_{10}>0$  \\\hline
				$C_{11}$&\multicolumn{2}{|c|}{ $(\operatorname{diam}(\varOmega))^{n+1-\beta+\gamma+\frac{2}{ab}\left(\alpha-\gamma+(1-b) s+(1-ab) t\right)}$ }   &$ 
			C_{11}>0$  \\\hline
				$C_{12}$&\multicolumn{2}{|c|}{ $C_{10}(a, b,\eta',\varepsilon, A,B,\alpha,\beta,\gamma,s,t, \delta,      \operatorname{\operatorname{diam}}(\varOmega), n)C_{11}(a,b,\alpha,\beta,\gamma,s,t,\operatorname{diam}(\varOmega),n) $ }   &$ 
			C_{12}>0$  \\\hline
		\end{tabular}
	\end{table}
	
	In view of \eqref{eq:W}, it is trivial to  verify the conditions of Lemma \ref{lemma:eigen} and hence we are able to use \eqref{eq:eigen-relation}:
	\[\frac{W_{r r} \cdot W_{n n}-\left|W_{r n}\right|^2}{W_{r r}+W_{n n}}\leq\lambda_-(D^2 W)\leq \min \left\{W_{r r}, W_{n n}\right\} \leq \max \left\{W_{r r}, W_{n n}\right\}\leq\lambda_+(D^2 W)\leq W_{r r}+W_{n n}.\]
	We  can conduct a series of calculations to derive
	\begin{equation*}
		W_{r r} \cdot W_{n n}-\left|W_{r n}\right|^2 \in\left[C_1(a, b, \delta) \cdot|W|^{2-b-a b} \cdot \xi^{-2}, C_2(a, b, \delta) \cdot|W|^{2-b-a b} \cdot \xi^{-2}\right],
	\end{equation*} 
	and    
	\begin{equation*}
		W_{r r}+W_{n n}\leq C_3(a, b, \delta, \operatorname{\operatorname{diam}}(\varOmega)) \cdot|W|^{1-a b} \cdot \xi^{-2}.
	\end{equation*}
	According to \eqref{eq:eigen-minmax}, it then follows that 
	\begin{align*}
		\lambda_{\min }(D^2 W)  & \geq\min \left\{\frac{W_r}{r}, \frac{W_{r r} \cdot W_{n n}-\left|W_{r n}\right|^2}{W_{r r}+W_{n n}}\right\}
		\geq C_4(a, b, \delta, \operatorname{\operatorname{diam}}(\varOmega)) \cdot|W|^{1-b},\\
		\lambda_{\max}(D^2 W)  & \geq W_{n n}  
		\geq C_5(a, b, \delta) \cdot|W|^{1-a b} \cdot \xi^{-2}.
	\end{align*}
	Thus by $(F_2)$, we can reach the  conclusion that
	\begin{align}\label{eq52}
		F\left(\lambda_1(D^2 W), \cdots, \lambda_n(D^2 W)\right) & \geq B\left(\lambda_{\min }(D^2 W)\right)^s\left(\lambda_{\max }(D^2 W)\right)^t \nonumber\\
		& \geq  B  C_4(a, b, \delta, \operatorname{\operatorname{diam}}(\varOmega))^s C_5(a, b, \delta)^t\cdot|W|^{(1-b) s+(1-a b) t} \cdot \xi^{-2 t}.
	\end{align}

	We can also infer that
	\begin{equation*}
		1+|DW|^2  =1+\left|W_r\right|^2+\left|W_n\right|^2=\left(1+\frac{\left|W_r\right|^2}{\left|W_n\right|^2}+\frac{1}{\left|W_n\right|^2}\right) \cdot\left|W_n\right|^2 
		\in\left[\left|W_n\right|^2,C_6(a,b,\delta,  \operatorname{\operatorname{diam}}(\varOmega) 
		) \cdot\left|W_n\right|^2\right],
	\end{equation*}
	where
	\begin{equation*}
		\left|W_n\right| \in\left[\frac{2}{a b}(1-\delta)^{\frac{a}{2}-1} \cdot|W|^{1-\frac{a b}{2}} \cdot \xi^{-1}, \frac{2}{a b} \cdot|W|^{1-\frac{a b}{2}} \cdot \xi^{-1}\right].
	\end{equation*}
	Hence,  it holds for any $\gamma \in \mathbb{R}$ that
	\begin{equation*}
		\left(1+|DW(x)|^2\right)^{-\frac{\gamma}{2}} \geq  C_7(a, b, \gamma, \delta, \operatorname{\operatorname{diam}}(\varOmega) 
		)  \cdot|W|^{\left(\frac{a b}{2}-1\right) \gamma} \cdot \xi^\gamma.
	\end{equation*}

By Remark \ref{remark:eta'}, we note that 
	$\varOmega_{1/2,P}$ is also of interior-$(a,\eta',\varepsilon)$-type at $P$. Hence  there exists a suitable domain  $V'\subseteq \varOmega_{1/2,P}$  as 
	\begin{equation}\label{def:V'}
		V' =\left\{x \in \mathbb{R}^n:\left|x^{\prime}\right|<\big(\frac{1}{4}\big)^{\frac{1}{a}} \varepsilon,\  \eta'\left|x^{\prime}\right|^a<x_n < \frac{1}{4} \eta'\varepsilon^a\right\}.
	\end{equation}
In order to get the bound of $\operatorname{dist}(x,\partial \varOmega)$ for any  $x\in V'$, we can adopt the same procedure of Lemma \ref{lemma:local-dx}:   first construct a domain $\widetilde{V}'$ satisfying $V'\subseteq\widetilde{V}'\subseteq \varOmega_{1/2,P}$ as  
	\[\widetilde{V}':=\left\{x \in \mathbb{R}^n:\left|x^{\prime}\right|<\varepsilon, \frac{1}{2} \eta'\left|x^{\prime}\right|^a<x_n<\frac{1}{2} \eta'\varepsilon^a\right\},\]
then show that for any $x\in V'\subseteq \widetilde{V}'$, there holds 
\begin{equation*}
\operatorname{dist}(x, \partial \varOmega)\in\left[2^{-1}\big(1+(a\eta' \varepsilon^{a-1})^2\big)^{-\frac{1}{2}}\cdot x_n,x_n\right].
\end{equation*}
Hence, for any $\beta\geq n+1+\gamma$ with $\gamma\in \mathbb{R}$, we have 
	\begin{equation*}
		\operatorname{dist}(x, \partial \varOmega)^{n+1-\beta}\geq C_8(a,\eta',\varepsilon,\beta,n)\cdot x_n{}^{n+1-\beta}.
	\end{equation*}

	Together with \eqref{eq-widetildef},  these estimates lead us to
	\begin{align}\label{eq60}
		{[\widetilde{f}(x, W, DW)]^{-1} } & = A^{-1} \operatorname{dist}(x, \partial \varOmega)^{n+1-\beta}|W|^\alpha\left(1+|DW|^2\right)^{-\frac{\gamma}{2}} \nonumber\\
		& \geq  A^{-1}   C_7(a, b, \gamma, \delta, \operatorname{\operatorname{diam}}(\varOmega) 
		) C_8(a,\eta',\varepsilon,\beta,n) \cdot|W|^{\alpha+(\frac{a b}{2}-1)\gamma}\cdot x_n{}^{n+1-\beta}\cdot \xi^{\gamma}.
	\end{align}
	By substituting (\ref{eq52}) and (\ref{eq60}) into \eqref{eq37}, we further deduce
	\begin{align*}
		\widetilde{H}[W]= & F\left(\lambda_1(D^2 W), \cdots, \lambda_n(D^2 W)\right) \cdot[\widetilde{f}(x, W, DW)]^{-1} \cdot \\
		\geq   & C_9(a, b, \eta',\varepsilon,A,B,\beta,\gamma,s,t, \delta,     \operatorname{\operatorname{diam}}(\varOmega), n) 
		\cdot|W|^{
		\alpha+(\frac{ab}{2}-1)\gamma+(1-b) s+(1-a b) t}\cdot x_n{}^{n+1-\beta}\cdot \xi^{\gamma-2t}.
	\end{align*}
By the definition \eqref{eq:W} of $W$, we have 
\begin{equation}\label{eq:estimate-W}
	|W|\in\left[(1-\delta)^{\frac{1}{b}}\cdot x_n{}^{\frac{2}{ab}}\cdot\xi^{-\frac{2}{ab}},x_n{}^{\frac{2}{ab}}\cdot\xi^{-\frac{2}{ab}}\right].
\end{equation}
It follows that 
\begin{align*}
	\widetilde{H}[W]&	\geq    C_{10}(a, b,\eta',\varepsilon, A,B,\alpha,\beta,\gamma,s,t, \delta,      \operatorname{\operatorname{diam}}(\varOmega), n) 	\cdot x_n{}^{n+1-\beta+\gamma+\frac{2}{ab}\left(\alpha-\gamma+(1-b) s+(1-ab) t\right)}\\
	&\ \ \ \ \cdot \xi^{-2t-\frac{2}{ab}\left(\alpha-\gamma+(1-b) s+(1-ab) t\right)}.\stepcounter{equation}\tag{\theequation}\label{eq:C10}
	\end{align*}

	We first let 
	\begin{equation}\label{conditionb-1-new}
		n+1-\beta+\gamma+\frac{2}{ab}\left(\alpha-\gamma+(1-b) s+(1-ab) t\right)\leq 0.
	\end{equation}
	Since $x_n\leq \operatorname{
		diam}(\varOmega)$, we have
	\[x_n{}^{n+1-\beta+\gamma+\frac{2}{ab}\left(\alpha-\gamma+(1-b) s+(1-ab) t\right)}\geq C_{11}(a,b,\alpha,\beta,\gamma,s,t,\operatorname{diam}(\varOmega),n).\]
	This further implies  
	\begin{equation*}
		\widetilde{H}[W]
		\geq C_{12}(a, b,\eta',\varepsilon, A,B,\alpha,\beta,\gamma,s,t, \delta,      \operatorname{\operatorname{diam}}(\varOmega), n)  \cdot \xi^{ -2t-\frac{2}{ab}\left(\alpha- \gamma+(1-b) s+(1-ab) t\right)}.
	\end{equation*}
 
	Now we set 
	\begin{equation}\label{conditionb-2-new}
		-2t-\frac{2}{ab}\left(\alpha- \gamma+(1-b) s+(1-ab) t\right)\leq 0.
	\end{equation}
	Then, by virtue of \eqref{assumption-7}, we can take
	\begin{equation}\label{eq:take-xi-1}
		0<\xi=\xi(a, \eta,\eta', A, B,\alpha, \beta,  \gamma,  s, t, \operatorname{\operatorname{diam}}(\varOmega), n)<1
	\end{equation}
	small enough such that
	\begin{equation}\label{eq61}
		\widetilde{H}[W] \geq 1 \text { in } V'.
	\end{equation}
	where we have imposed two condtions \eqref{conditionb-1-new} and \eqref{conditionb-2-new}, i.e.  $b$ satisfies
	\begin{equation} \label{eq:take-b-1}
	\frac{2(\alpha-\gamma+s+t)}{a(\beta-\gamma-n-1+2t)+2s}\leq b\leq \frac{\alpha-\gamma+s+t}{s}.
	\end{equation}
For convenience, we can take 
\begin{equation*}
b=b_0:=\frac{2(\alpha-\gamma+s+t)}{a(\beta-\gamma-n-1+2t)+2s}.
\end{equation*}
	Here, we remark  that $2 \leq a<\frac{2(\alpha-\gamma+t)}{\beta-\gamma+2t-n-1}$ gives $b_0>1$ for \textit{Step 1-1} while  $\frac{2(\alpha-\gamma+t)}{\beta-\gamma+2t-n-1} \leq a<+\infty$ gives $0< b_0\leq 1$ for \textit{Step 1-2}.

	Up to now, 
	by (\ref{eq31}), \eqref{eq:HW}, \eqref{eq:tildeHW} and (\ref{eq61}), we conclude that $W_{\operatorname{sub}}$ is a subsolution to the problem (\ref{eq01})-\eqref{eq01boundary} over $V'$, where $b$ and $\xi$ satisfy \eqref{eq:take-b-1} and \eqref{eq:take-xi-1} respectively.

	\subsection{Final  estimate on upper bound}

	Based on  the construction of subsolution $W_{\operatorname{sub}}$ in the last subsection together with the assumption \eqref{eq:assumption}, we are able to use the comparison principle to derive that  
	\begin{equation}\label{eq:comparison}
		0 \geq u\geq W\ \ \text{ in } V'.
	\end{equation}
	Here we notice that we can take  $\eta'\varepsilon^{a}=\operatorname{diam}(\varOmega_{1/2,P})$ such that $\left(\frac{1}{4}\varOmega_{1/2,P}\right) \cap x_n\text{-axis} = V'\cap x_n\text{-axis} $.

	We first restrict \eqref{eq:comparison} onto the $x_{n}$-axis.  Taking $y=(0,y_n)$ now gives $y_n=d_y=\operatorname{dist}(y,P)$. 
	By \eqref{assumption-3}, \eqref{assumption-5} and \eqref{eq:W}, it follows that $u(P)=u(0)=\varphi(0)=0$ and $W(P)=W(0)=0$. Thus, by \eqref{eq:W} again, we have
	\begin{align*}
		|u(y)-u(P)| & \leq|W(y)-W(P)|=|W(y)|\\ &=\left(\frac{y_n}{\xi(a, \eta, A,B, \alpha, \beta, \gamma, s, t, \operatorname{\operatorname{diam}}(\varOmega), n)}\right)^{\frac{2}{a b}} \\
		&= C(a, \eta, A,B, \alpha, \beta, \gamma, s, t, \operatorname{\operatorname{diam}}(\varOmega), n) 
		\left(\operatorname{dist}(y,P)\right)^{\mu(a)}.
	\end{align*}
	Here and in what follows, we have used   the fact that 
	\[ \mu(a)=\frac{\beta-\gamma+2t-n-1}{\alpha-\gamma+s+t}+\frac{2 s}{a(\alpha-\gamma+s+t)}=\frac{2}{ab}\]
	due to \eqref{assumption-1'}  and \eqref{eq:take-b-1}.
	
	Next we consider an arbitrary point  $y\in V'$ that is not on the $x_{n}$-axis.  Let $\theta(y)$ denote  the angle between $l_{yP}$ and the $x_{n}$-axis. Noting 
	$y_{n} = d_{y}\cos\left(\theta(y)\right)= \operatorname{dist}(y,P)\cos\left(\theta(y)\right)$, we deduce from \eqref{eq:comparison} and \eqref{eq:W} that
	\begin{align*}
		|u(y)-u(P)| & \leq|W(y)-W(P)|=|W(y)|\\ 
		&\leq \left(\frac{y_{n}}{\xi(a, \eta, A,B, \alpha, \beta, \gamma, s, t, \operatorname{\operatorname{diam}}(\varOmega), n)}\right)^{\frac{2}{a b}}\\
		& =C(a, \eta, A,B, \alpha, \beta, \gamma, s, t, \operatorname{\operatorname{diam}}(\varOmega), n) \big(\operatorname{dist}(y,P)\cos\left(\theta(y)\right)\big){}^{\mu(a)} \\
		& \leq C(a, \eta, A,B, \alpha, \beta, \gamma, s, t, \operatorname{\operatorname{diam}}(\varOmega), n) \left(\operatorname{dist}(y,P)\right)^{\mu(a)}.
	\end{align*}

		We are now in a position to consider $y\in \varOmega_{1/2,P}\setminus V'$. 
Since there exists at least one  point in $V'\cap l_{yP}$, we can fix $y'\in V'\cap l_{yP}$. It follows that $\operatorname{dist}(y',P)\leq \operatorname{dist}(y,P)$ and then
	\begin{align*}
		|u(y)-u(P)| 
			&\leq C(a, \eta, A,B, \alpha, \beta, \gamma, s, t, \operatorname{\operatorname{diam}}(\varOmega), n) \left(\operatorname{dist}(y',P)\right)^{\mu(a)}\\
		&\leq C(a, \eta, A,B, \alpha, \beta, \gamma, s, t, \operatorname{\operatorname{diam}}(\varOmega), n) \left(\operatorname{dist}(y,P)\right)^{\mu(a)}.
	\end{align*}

	Therefore, for any $y\in\varOmega_{1/2,P}$, we always have 
	\begin{equation}\label{eq:Holder}
		|u(y)-u(P)|
		\leq  C(a, \eta, A,B, \alpha, \beta, \gamma, s, t, \operatorname{\operatorname{diam}}(\varOmega), n) \left(\operatorname{dist}(y,P)\right)^{\mu(a)}.
	\end{equation}
	This inequality holds not only in the local region $\varOmega_{1/2,P}$ but actually throughout the entire domain  $\varOmega$.  
	Indeed, for any  $y\in\varOmega\setminus\varOmega_{1/2,P}$, we can always find $P'\in \partial\varOmega$ such that $y\in\varOmega_{1/2,P'}$, and hence we can repeat the above procedure to obtain \eqref{eq:Holder}  when replacing $P$ by $P'$ and $y\in\varOmega_{1/2,P}$ by $y\in\varOmega_{1/2,P'}$. Together with the arbitrariness of $P\in\partial\varOmega$, it implies that
	\begin{equation}\label{eq:Holder-final}
		|u(y)-u(x)|
		\leq  C(a, \eta, A,B, \alpha, \beta, \gamma, s, t, \operatorname{\operatorname{diam}}(\varOmega), n) \left(\operatorname{dist}(y,x)\right)^{\mu(a)},\ \ \forall x\in\partial\varOmega,\ \forall y\in \varOmega_{1/2,x}.
	\end{equation}
	Since $\varOmega$ is of exterior-$(a,\eta)$-type, we get $a_{\max}=a$ and also obtain 
	\begin{equation}\label{conclusion1}
		|u(y)-u(x)|
		\leq  C(a, \eta, A,B, \alpha, \beta, \gamma, s, t, \operatorname{\operatorname{diam}}(\varOmega), n) \left(\operatorname{dist}(y,x)\right)^{\mu(a_{\max})},\ \ \forall x\in\partial\varOmega,\ \forall y\in \varOmega_{1/2,x}.
	\end{equation}
	This completes the proof of  case $a\in[2,+\infty)$ for Theorem \ref{thm1-1} and the upper bound part in  Theorem \ref{thm1}.

	Now we consider the case that both $x,y\in\overline{\varOmega}$. Then there exist $x_1,x_2\in\partial \varOmega$ such that $l_{xy}\cap \partial \varOmega = \{x_1,x_2\}$.  Denote $  \partial \varOmega_{1/2,x_1}  \cap \partial \varOmega_{1/2,x_2} = \{y_0\}$. In fact, we also have  $l_{xy} \cap \partial \varOmega_{1/2,x_1} = l_{xy} \cap \partial \varOmega_{1/2,x_2} = \{y_0\}$, and then 
	$$u(y_0) = \min_{ l_{xy}}u.$$
	Therefore we get 
	\begin{align}
		|u(y) - u(x)| & \leq \max\{|u(y_0) - u(y)|,|u(y_0) - u(x)|\} \\ \nonumber
		&\leq  C(a, \eta, A,B, \alpha, \beta, \gamma, s, t, \operatorname{\operatorname{diam}}(\varOmega), n) \max\{\left(\operatorname{dist}(y_0,y)\right)^{\mu(a_{\max})},\left(\operatorname{dist}(y_0,x)\right)^{\mu(a_{\max})}\} \\ \nonumber
		&\leq  C(a, \eta, A,B, \alpha, \beta, \gamma, s, t, \operatorname{\operatorname{diam}}(\varOmega), n) \left(\operatorname{dist}(y,x)\right)^{\mu(a_{\max})} \nonumber ,\ \ \forall x,y\in \overline{\varOmega},
	\end{align}
	and thus		$$ u \in C^{\mu(a_{\max})}(\overline{\varOmega}).$$

	Alternatively, \eqref{conclusion1} also implies 
	\[\text{$\varOmega$ is an upper-$(\mu(a_{\max}),C(a, \eta, A,B, \alpha, \beta, \gamma, s, t, \operatorname{\operatorname{diam}}(\varOmega), n))$ type domain about  $u$,}\]
	then by Lemma \ref{lemma-new},  we obtain 
	$$ u \in C^{\mu(a_{\max})}(\overline{\varOmega}),$$
	and
	\begin{equation*}
		\Vert u\Vert _{C^{\mu(a_{\max})}(\overline{\varOmega})} \leq C(a_{\max}, \eta, A,B, \alpha, \beta, \gamma, s, t, \operatorname{\operatorname{diam}}(\varOmega), n)\left(\left(\operatorname{diam}(\varOmega)\right)^{\mu(a_{\max})}+1\right)+\mathop{\max}_{\partial\varOmega} |\varphi|.
	\end{equation*}
	The proof of Theorem \ref{thm1-1'}   is now complete.

	\section{Boundary lower bound estimate}\label{sec:super}
	
	\subsection{Reduction of the problem}
	
For the case $a\in[2,+\infty)$, it now remains to prove the left hand side of the inequality (\ref{eq09}). 
	
	To establish this, we need boundedness conditions $(F_2')$ and $(f_3')$, which are similar to $(F_2)$ and $(f_3)$ in the last section but with opposing symbols.
	At this point, we can restate the main result of lower bound estimate as the following theorem, which particularly aims  to determine the  order of lower bound.

	\begin{theorem}\label{thm3}
		Suppose that  $ \varOmega \subseteq \mathbb{R} ^{n} $ 
		is a bounded convex domain and $P \in \partial \varOmega$ is an interior-$(a, \eta, \varepsilon)$-type point with $a \in[2,+\infty)$. Without loss of generality, we assume that
		$$
		P=0 \quad \text { and } \quad\left\{x \in \mathbb{R}^n:\left|x^{\prime}\right|<\mu, \frac{1}{2} \eta\left|x^{\prime}\right|^a<x_n<\frac{1}{2} \eta\varepsilon^a\right\} \subseteq \varOmega \subseteq \mathbb{R}_{+}^n.
		$$
		We also assume that 
		$F$ satisfies $(F_1)$  and $(F_2')$,  $f$ satisfies $(f_1)$, $(f_2)$ and $(f_3')$, $\varphi \in C(\partial\varOmega)$, and $\mu$ satisfies \eqref{eq10}. 
		If $u$ is a convex viscosity solution to the problem \eqref{eq01}-\eqref{eq01boundary}, 
		then for any point $Q \in \varOmega_{1/2,P}$, there exists some positive constant  $C$ (depending on $a, b,\eta,\varepsilon,\xi, A,B,\alpha,\beta,\gamma,s,t, n$) 
		such that 
		$$
		\left|u\left(x\right) - u(P)\right| \geq C(a, b,\eta,\varepsilon,\xi, A,B,\alpha,\beta,\gamma,s,t, n) \left(\operatorname{dist}(x,P)\right)^{\frac{2}{a b}},$$
		where 
		\[b\leq\frac{2(\alpha-\gamma+s+t)}{a(\beta-\gamma-n-1+2t)+2s}.\]
	\end{theorem}

	\begin{remark}
		For a special case of Theorem \ref{thm3}, we refer the readers to Theorem 1.1 in  \cite{Li-Li-F}  which addressed the issue of zero boundary value problems.
	\end{remark}

	We recall that the work \cite{Li-Li-F} discussed  the zero boundary value problem in the subdomain $V$ since the  problem only have local properties. 
	To study the issue in nonzero boundary value situation, we primarily consider the domain $V \cap \varOmega_{1/2,P}$ in this subsection. 
	However, it is rather difficult to directly construct a supersolution to our problem  \eqref{eq01}-\eqref{eq01boundary} over $V \cap \varOmega_{1/2,P}$ (since the boundary of this domain is unknown). Nevertheless, this issue will be solved if we have a supersolution to \eqref{eq01}-\eqref{eq01boundary}  over $V$. Our goal has now become to first search for two suitable positive constants $b$ and $\xi$ in \eqref{eq:W} to make $W$ a supersolution   over $V$, and then construct the final upper estimate over $\varOmega_{1/2,P}$.
	
	\begin{remark}\label{remark:V}
		In the following proof, we indeed have   
		$V\cap\varOmega_{1/2,P}=V$ since we can take $\eta$   sufficiently large and  $\varepsilon$
		as stated in  Remark \ref{remark:etaepsilon}. 
	\end{remark}
	
	\subsection{Construction of supersolution}
Consider  
\begin{equation*}
	W_{\operatorname{sup}}:=W+\varphi_*.
\end{equation*}
Since $\varphi_*=\varphi=u$ on $\partial\varOmega$ and $\varphi_*$ is concave, we have $\varphi_*\geq u$ over $\varOmega$. 
It is  clear that 
\begin{equation*}
	W=0\ \ \text{ on }S,
\end{equation*}
and then 
\begin{equation*}
	W_{\operatorname{sup}}=\varphi_*\geq u\ \ \text{ on }S.
\end{equation*}

Since the solution $u$ is a convex function and satisfies $u(P)=0$, it follows that $u\leq 0$ in $\varOmega_{1/2,P}$ and moreover $L\cap \varOmega_{1/2,P}\subseteq\varOmega_{1/2,P}$ gives $u\leq 0$ on $L$. To exclude the case $u\equiv0$ in $\varOmega$, we further obtain  $u<0$ on $L$, which implies that there exists some constant $m$ such that 
\[\max _{L} u \leq-m-|\varphi_*|_{C^{\mu(a)}(\overline{\varOmega})}+\min_{L}\varphi_*<0.\]  
By \eqref{eq:W} (the definition of $W$), we note that
\begin{equation*}
	\min _{L} W=W\big(0, \frac{1}{4} \eta\varepsilon^a\big)=-\big(\frac{1}{4}\eta\varepsilon^a\big)^{\frac{2}{a b}} \cdot \xi^{-\frac{2}{ab}}.
\end{equation*} 
In order to let $W_{\operatorname{sup}}\geq u$ on $L$, we just need 
\begin{equation}\label{eq4.8}
	\xi \geq \xi_0:=\frac{1}{4}\eta\varepsilon^a \left(m+|\varphi_*|_{C^{\mu(a)}(\overline{\varOmega})}\right)^{-\frac{ab}{2}},
\end{equation}
which yields
\begin{equation*}
	W_{\operatorname{sup}}\geq \min_L W+\min_{L}\varphi_*\geq -m-|\varphi_*|_{C^{\mu(a)}(\overline{\varOmega})}+\min_{L}\varphi_*\geq \max_L u\geq u\  
	\text{ on }L.
\end{equation*}

Since $\partial V=L\cup S$,
combining these estimates shows that 
\begin{equation}\label{eq:Hboundary}
	W_{\operatorname{sup}} \geq u \ \ \text{ on }\partial V.
\end{equation}

The rest of this subsection is devoted to constructing 
	\begin{equation*}
	W_{\operatorname{sup}}:=W+\varphi_*
\end{equation*}
as a supersolution to  \eqref{eq01} in $V$,   
 that is, to satisfy
\begin{equation}\label{eq:HW'}
	H[W_{\operatorname{sup}}] \leq 1\ \ \text{ in }V.
\end{equation}
According to the condition $(f_3')$ and the fact that $\varphi_*$ is concave, we note that 
\begin{align*}
	H[W_{\operatorname{sup}}]&=F(\lambda_1(D^2W_{\operatorname{sup}}),\cdots,\lambda_n(D^2W_{\operatorname{sup}}))\cdot [f(x,W_{\operatorname{sup}},DW_{\operatorname{sup}})]^{-1}\\
	&=F(\lambda_1(D^2W+D^2\varphi_*),\cdots,\lambda_n(D^2W+D^2\varphi_*))\cdot [f(x,W+\varphi_*,DW+D\varphi_*)]^{-1}\\
	&\leq F(\lambda_1(D^2W),\cdots,\lambda_n(D^2W))\cdot \left[A^{-1} \operatorname{dist}(x, \partial \varOmega)^{n+1-\beta}|W|^\alpha\left(1+|DW|^2\right)^{-\frac{\gamma}{2}}\right]^{-1} \\
	&=F(\lambda_1(D^2W),\cdots,\lambda_n(D^2W))\cdot[\widetilde{f}(x,W,DW)]^{-1}=\widetilde{H}[W].
\end{align*}
Now it suffices to satisfy
\begin{equation}\label{eq:tildeHW'}
	\widetilde{H}[W] \leq 1\ \ \text{ in }V.
\end{equation}
For convenience, the relevant constants appearing in what follows are collected in Table \ref{table:constant-2}.

\begin{table}[ht] \caption{Collection of constants in upper bound estimate for the case $a\in [2,+\infty)$}
	\label{table:constant-2}
	\renewcommand{\arraystretch}{1.8}\centering
	\begin{tabular}{|c|c|c|c|c|} 
		\hline
		Symbol & \multicolumn{2}{|c|}{Values}
		& Sign
		\\ [0.5ex]
		\hline
		$C_{13}$& \multicolumn{2}{|c|}{$\frac{2}{b}$}  &$C_{13}>0$ \\\hline 
		$C_{14}$& \multicolumn{2}{|c|}{$2^{\frac{6}{a}-4}\left( \frac{(n-1)}{b}+\frac{2|b-1|}{b^2}\right)\eta^2\varepsilon^{2a-2} +\left(\frac{2|a-2|}{a^2 b}+\frac{4|b-1|}{a^2 b^2}\right)  $}  &$C_{14}>0$ \\\hline
		$C_{15}$& \multicolumn{2}{|c|}{$\max \Big\{2^{\beta-n-1}\big(1+(a\eta \varepsilon^{a-1})^2\big)^{\frac{\beta-n-1}{2}}, 1\Big\}$ }  &$C_{15}>0$  \\\hline
		$C_{16}$& \multicolumn{2}{|c|}{$4\big(\rho^{-1}+1\big)^2\big(1+(2a\eta \varepsilon^{a-1})^2\big) \max\Big\{\Big(\big(1-\big(\frac{1}{4}\big)^{\frac{2}{a}}\big)\Big)^{2-a},1\Big\}$}  &$C_{16}>0$\\\hline
		$C_{17}$& \multicolumn{2}{|c|}{$\max\left\{2^{-\gamma}a^\gamma b^\gamma,C_{16}(a,b,\eta,\varepsilon)^{-\frac{\gamma}{2}}\right\}$ }  &$C_{17}>0$  \\\hline
		$C_{18}$& \multicolumn{2}{|c|}{$A^{-1}BC_{13}(b)^{s}C_{14}(a,b,\eta,\varepsilon,n)^{t}C_{15}(a,\eta,\varepsilon,\beta,n)C_{17}(a,b,\eta,\varepsilon)$  } &$C_{18}>0$  \\\hline
		$C_{19}$& \multicolumn{2}{|c|}{$\max\Big\{\big(1-\big(\frac{1}{4}\big)^{\frac{2}{a}}\big)^{\frac{1}{b}\left(\alpha+(b-1) \gamma+(1-b) s+(1-ab) t\right)},1\Big\}$ }  &$C_{19}>0$  \\\hline
		$C_{20}$& \multicolumn{2}{|c|}{$C_{18}(a, b,\eta,\varepsilon, A,B,\beta,\gamma,s,t, n)C_{19}(a,b,\alpha,\gamma,s,t) $}   &$C_{20}>0$  \\\hline
		$C_{21}$& \multicolumn{2}{|c|}{$(\frac{1}{4}\eta\varepsilon^a)^{n+1-\beta+\gamma+\frac{2}{ab}\left(\alpha-\gamma+(1-b) s+(1-ab) t\right)}$}   &$C_{21}>0$  \\\hline
		$C_{22}$& \multicolumn{2}{|c|}{$C_{20}(a, b,\eta,\varepsilon, A,B,\alpha,\beta,\gamma,s,t, n)  C_{21}(a,b,\eta,\varepsilon,\alpha,\beta,\gamma,s,t,n)$}   &$C_{22}>0$  \\\hline
	\end{tabular}
\end{table}

We notice that both $\eta$ and $\xi$ can be taken as sufficiently large numbers (see Remark \ref{remark:etaepsilon} for $\eta$). Let us just take  $\xi=\lambda\eta$, where $\lambda>0$. 
For example, we let  $\lambda=\big(\frac{1}{2}\big)^{\frac{a}{2}-1}$, i.e.
\begin{equation}\label{eq:take-xi-2'}
	\xi=\big(\frac{1}{2}\big)^{\frac{a}{2}-1}\eta
\end{equation}
for the convenience of the following analysis.

By virtue of \eqref{eq:eigen-minmax}, \eqref{eq:take-xi-2'} and the fact $x_n\leq \frac{1}{4}\eta\varepsilon^a$, we can deduce that 
\begin{align*}
	\lambda_{\min }(D^2 W)  & \leq\frac{W_r}{r} 
	=C_{13}(b)\cdot|W|^{1-b},\\   
	\lambda_{\max}(D^2 W)  & \leq(n-2)\frac{W_r}{r}+W_{rr}+W_{nn} 
	\leq C_{14}(a,b,\eta,\varepsilon,n)\cdot|W|^{1-ab}\cdot  \xi^{-2}.
\end{align*}
According to $(F_2')$, it then follows that 
\begin{align*}
	F\left(\lambda_1\left(D^2 W\right), \cdots, \lambda_n\left(D^2 W\right)\right) &\leq B(\lambda_{\min}(D^{2}u))^{s}(\lambda_{\max}(D^{2}u))^{t} \\
	&= BC_{13}(b)^{s}C_{14}(a,b,\eta,\varepsilon,n)^{t} \cdot|W|^{(1-b) s+(1-ab) t}\cdot  \xi^{-2t}.\stepcounter{equation}\tag{\theequation}\label{eq:F}
\end{align*}

Now we refer to the geometric handling in \cite{Li-Li-F} to bound $\operatorname{dist}(x, \partial \varOmega)$ and $1+|D W|^2$ for any $x\in V$: firstly, we have Lemma \ref{lemma:local-dx} and hence  for any  $\beta \in \mathbb{R}$, there holds
\begin{equation}\label{eq:dx}
	\operatorname{dist}(x, \partial \varOmega)^{n+1-\beta} \leq  
	C_{15}(a,\eta,\varepsilon,\beta,n)\cdot x_n{ }^{n+1-\beta};
\end{equation}
secondly, we have 
\begin{equation}\label{eq:DW}
	1+|DW|^2\in \left[\frac{4}{a^2b^2}\cdot|W|^{2-2b}\cdot x_n{}^{\frac{4}{a}-2}\cdot\xi^{-\frac{4}{a}},4\big(\rho^{-1}+1\big)^2\big(1+(2a\eta \varepsilon^{a-1})^2\big)\cdot|W|^{2-ab}\cdot\xi^{-2}\right],
\end{equation}
where the constant  $\rho=\rho(a, \eta, \varepsilon)>0$ is taken so that  
\begin{equation*}
	|W|\geq \rho\operatorname{dist}(x,S),\ \ \forall  x\in V.
\end{equation*}
Here, we have omitted the details of these bounds.

Before deriving the bounds on $(1+|DW|^2)^{-\gamma}$ for any $\gamma\in \mathbb{R}$, we first unify the form of the bounds on $1+|DW|^2$ based on  \eqref{eq:DW}. This can be finished by showing the following claim: 
\begin{equation}\label{eq:W-bound}
	|W|\in\left[\big(1-\big(\frac{1}{4}\big)^{\frac{2}{a}}\big)^{\frac{1}{b}}\cdot\big(\frac{x_n}{\xi}\big)^{\frac{2}{ab}},\big(\frac{x_n}{\xi}\big)^{\frac{2}{ab}}\right],\ \ \forall  x\in V.
\end{equation}
In fact, the upper bound is trivial to see from \eqref{eq:W} (the definition of $W$),   
and hence we only need to derive the lower bound. 
This is clear when $x_n=0$. It remains to consider  $x_n\in (0,C_{\eta,\varepsilon}]$. Let $x_n=\frac{kC_{\eta,\varepsilon}}{4}$, where   $k\in (0,4]$. 
By \eqref{eq:V} (the definition of $V$) and \eqref{eq4.10}, we get
\[r^2\leq \big(\frac{1}{4}\big)^{\frac{2}{a}}\varepsilon^2=\big(\frac{1}{4}\big)^{\frac{2}{a}}\big(\frac{C_{\eta,\varepsilon}}{\eta}\big)^{\frac{2}{a}}=\big(\frac{C_{\eta,\varepsilon}}{4\eta}\big)^{\frac{2}{a}}=\big(\frac{x_n}{k\xi}\big)^{\frac{2}{a}}.\]
Using \eqref{eq:W} again gives 
\begin{equation*}
	|W|=\Big|\big(\frac{x_n}{\xi}\big)^{\frac{2}{a}}-r^2\Big|^{\frac{1}{b}}\geq  \Big|\big(\frac{x_n}{\xi}\big)^{\frac{2}{a}}-\big(\frac{x_n}{k\xi}\big)^{\frac{2}{a}}\Big|^{\frac{1}{b}}.
\end{equation*}
For the case $k\in [2^{-\frac{a}{2}},4]$, we note that 
\begin{align*}
	\Big|\big(\frac{x_n}{\xi}\big)^{\frac{2}{a}}-\big(\frac{x_n}{k\xi}\big)^{\frac{2}{a}}\Big|^{\frac{1}{b}}=\Big(1-\big(\frac{1}{k}\Big)^{\frac{2}{a}}\big)^{\frac{1}{b}}\big(\frac{x_n}{\xi}\big)^{\frac{2}{a}}\geq \min\left\{\big(1-\big(\frac{1}{4}\big)^{\frac{2}{a}}\big)^{\frac{1}{b}},1\right\}\cdot\big(\frac{x_n}{\xi}\big)^{\frac{2}{ab}}=\big(1-\big(\frac{1}{4}\big)^{\frac{2}{a}}\big)^{\frac{1}{b}}\cdot\big(\frac{x_n}{\xi}\big)^{\frac{2}{ab}}.
\end{align*}
For the case $k\in (0,2^{-\frac{a}{2}}]$, 
we can obtain 
\begin{align*}
	\Big|\big(\frac{x_n}{\xi}\big)^{\frac{2}{a}}-\big(\frac{x_n}{k\xi}\big)^{\frac{2}{a}}\Big|^{\frac{1}{b}}=\Big(\big(\frac{1}{k}\big)^{\frac{2}{a}}-1\Big)^{\frac{1}{b}}\big(\frac{x_n}{\xi}\big)^{\frac{2}{a}}\geq  \big(\frac{x_n}{\xi}\big)^{\frac{2}{ab}}.
\end{align*}
Summing up these two cases, we have
\[|W|\geq \big(1-\big(\frac{1}{4}\big)^{\frac{2}{a}}\big)^{\frac{1}{b}}\cdot\big(\frac{x_n}{\xi}\big)^{\frac{2}{ab}},\ \ \forall  x\in V. \]
Thus the proof of \eqref{eq:W-bound} has been completed. 
We now return to unify the form of bounds in \eqref{eq:DW}. In view of \eqref{eq:W-bound}, 
we deduce  that 
\[
|W|^{(2-a)b}\leq \max\left\{\Big(\big(1-\big(\frac{1}{4}\big)^{\frac{2}{a}}\big)\Big)^{2-a},1\right\}\cdot x_n{}^{\frac{4}{a}-2}\cdot\xi^{2-\frac{4}{a}}.\]
and hence 
\[
|W|^{2-ab}\leq \max\left\{\Big(\big(1-\big(\frac{1}{4}\big)^{\frac{2}{a}}\big)\Big)^{2-a},1\right\}\cdot|W|^{2-2b}\cdot x_n{}^{\frac{4}{a}-2}\cdot\xi^{2-\frac{4}{a}}.\]
Hence we can infer from  \eqref{eq:DW} that 
\begin{equation}\label{eq:DW-1}
	1+|DW|^2\in \left[\frac{4}{a^2b^2}\cdot|W|^{2-2b}\cdot x_n{}^{\frac{4}{a}-2}\cdot\xi^{-\frac{4}{a}},C_{16}(a,b,\eta,\varepsilon)\cdot|W|^{2-2b}\cdot x_n{}^{\frac{4}{a}-2}\cdot\xi^{-\frac{4}{a}}\right].
\end{equation}

Up to now, it holds for any $\gamma \in \mathbb{R}$ that
\begin{equation}\label{eq:DWgamma}
	\left(1+|DW|^2\right)^{-\frac{\gamma}{2}} \leq C_{17}(a,b,\eta,\varepsilon)  \cdot|W|^{(b-1) \gamma}\cdot x_n{}^{(1-\frac{2}{a})\gamma}\cdot \xi^{\frac{2}{a}\gamma}.
\end{equation}
It is time to substitute \eqref{eq:dx} and \eqref{eq:DWgamma} into \eqref{eq-widetildef} to derive 
\begin{align*} 
	{[\widetilde{f}(x, W, DW)]^{-1} } &= A^{-1} \operatorname{dist}(x, \partial \varOmega)^{n+1-\beta}|W|^{\alpha}\left(1+|DW|^2\right)^{-\frac{\gamma}{2}} \\
	& \leq   
	A^{-1}C_{15}(a,\eta,\varepsilon,\beta,n)C_{17}(a,b,\eta,\varepsilon)\cdot |W|^{\alpha+(b-1) \gamma}\cdot x_n{}^{n+1-\beta+(1-\frac{2}{a})\gamma}\cdot \xi^{\frac{2}{a}\gamma}.\stepcounter{equation}\tag{\theequation}\label{eq:f}
\end{align*}

Putting (\ref{eq:F}) and (\ref{eq:f}) back into \eqref{eq37} now leads us to 
\begin{align*}
	\widetilde{H}[W]= & F\left(\lambda_1(D^2 W), \cdots, \lambda_n(D^2 W)\right) \cdot[\widetilde{f}(x, W, DW)]^{-1} \cdot \\
	\leq   & C_{18}(a, b,\eta,\varepsilon, A,B,\beta,\gamma,s,t, n)\cdot   |W|^{\alpha+(b-1) \gamma+(1-b) s+(1-ab) t}\cdot x_n{}^{n+1-\beta+(1-\frac{2}{a})\gamma}\cdot \xi^{\frac{2}{a}\gamma-2t}.
\end{align*}
We note that \eqref{eq:W-bound} yields
\begin{align*}
	|W|^{\alpha+(b-1) \gamma+(1-b) s+(1-ab) t}&\leq C_{19}(a,b,\alpha,\gamma,s,t)   
	\cdot x_n{}^{\frac{2}{ab}\left(\alpha+(b-1) \gamma+(1-b) s+(1-ab) t\right)} 
\cdot \xi{}^{-\frac{2}{ab}\left(\alpha+(b-1) \gamma+(1-b) s+(1-ab) t\right)}.
\end{align*}
It follows that 
\begin{align*}
	\widetilde{H}[W]
	&\leq    C_{20}(a, b,\eta,\varepsilon, A,B,\alpha,\beta,\gamma,s,t, n)\cdot    x_n{}^{n+1-\beta+\gamma+\frac{2}{ab}\left(\alpha-\gamma+(1-b) s+(1-ab) t\right)} 
	\cdot \xi^{-2t-\frac{2}{ab}\left(\alpha- \gamma+(1-b) s+(1-ab) t\right)}.
\end{align*}

We first let 
\begin{equation}\label{conditionb-1}
	n+1-\beta+\gamma+\frac{2}{ab}\left(\alpha-\gamma+(1-b) s+(1-ab) t\right) \geq  0.
\end{equation}
Based on the fact $x_n\leq \frac{1}{4}\eta\varepsilon^a$, we obtain
\[x_n{}^{n+1-\beta+\gamma+\frac{2}{ab}\left(\alpha-\gamma+(1-b) s+(1-ab) t\right)}\leq C_{21}(a,b,\eta,\varepsilon,\alpha,\beta,\gamma,s,t,n).\]
This further implies that 
\begin{equation*}
	\widetilde{H}[W]
	\leq    C_{22}(a, b,\eta,\varepsilon, A,B,\alpha,\beta,\gamma,s,t, n)   \cdot \xi^{-2t-\frac{2}{ab}\left(\alpha- \gamma+(1-b) s+(1-ab) t\right)}.
\end{equation*}
Now we set 
\begin{equation}\label{conditionb-2}
	-2t-\frac{2}{ab}\left(\alpha- \gamma+(1-b) s+(1-ab) t\right)\leq 0.
\end{equation}
Therefore, we can take 
\begin{equation}\label{eq:take-xi-2}
	\xi=\xi(a, b,\eta,\varepsilon, A,B,\alpha,\beta,\gamma,s,t, n)\geq\xi_0
\end{equation}
large enough (see the definition of $\xi_0$ in \eqref{eq4.8})  such that
\begin{equation}\label{eq:H}
	\widetilde{H}[W]
	\leq 1 \ \ \text{ in }V,
\end{equation}
where we have imposed two condtions \eqref{conditionb-1} and \eqref{conditionb-2}, i.e.  $b$ satisfies
\begin{equation}\label{eq:take-b-2}
	b\leq\min\left\{\frac{2(\alpha-\gamma+s+t)}{a(\beta-\gamma-n-1+2t)+2s},\frac{\alpha-\gamma+s+t}{s}\right\}=\frac{2(\alpha-\gamma+s+t)}{a(\beta-\gamma-n-1+2t)+2s}.
\end{equation}
For convenience, we can also take 
\begin{equation*}
	b=b_0=\frac{2(\alpha-\gamma+s+t)}{a(\beta-\gamma-n-1+2t)+2s}.
\end{equation*}

In summary, putting 
\eqref{eq:Hboundary} and (\ref{eq:H}) together shows that $W_{\operatorname{sup}}$ is a supersolution to the problem (\ref{eq01})-\eqref{eq01boundary} over $V$.  
Here, $b$ satisfies \eqref{eq:take-b-2}, and  $\xi$ satisfy \eqref{eq:take-xi-2'} and  \eqref{eq:take-xi-2}.

	\subsection{Final  estimate on lower bound}

	Using the supersolution $W_{\operatorname{sup}}$ constructed in the previous subsection, along with the assumption \eqref{eq:assumption}, we can apply the the comparison principle to conclude that 
	\begin{equation}\label{eq:comparison-2}
		u\leq W\leq 0 \ \ \text{ in }V.
	\end{equation}
	By \eqref{eq:W}, we can further deduce 
	\begin{equation}\label{eq:comparison-3}
		|u(y)-u(P)|  \geq|W(y)-W(P)|=|W(y)|=\left|\left(\frac{y_n}{\xi}\right)^{\frac{2}{a}}-|y'|^2\right|^{\frac{1}{b}},\ \ \ \forall y\in V.
	\end{equation}

	Now we construct the following subdomain $V_0\subseteq V$:
	\begin{equation*}
		V_0:=\left\{x \in \mathbb{R}^n:\left|x^{\prime}\right|<\big(\frac{1}{8}\big)^{\frac{1}{a}} \varepsilon,\  2\eta\left|x^{\prime}\right|^a<x_n < \frac{1}{4} \eta\varepsilon^a\right\}.
	\end{equation*}
	For any $y\in V_0$, it is apparent that $2\eta\left|y^{\prime}\right|^a<y_n$ and hence \eqref{eq:take-xi-2'} gives
	\[|y'|^2<\big(\frac{1}{2}\big)^{\frac{2}{a}}\big(\frac{y_n}{\eta}\big)^{\frac{2}{a}}=\frac{1}{2}\big(\frac{y_n}{\xi}\big)^{\frac{2}{a}}.\]
	Taking this back to \eqref{eq:comparison-3} then shows 
	\begin{equation}\label{eq:compariosn-4}
		|u(y)-u(P)|  \geq
		\big(\frac{1}{2}\big)^{\frac{1}{b}} \left(\frac{y_n}{\xi}\right)^{\frac{2}{ab}}\geq C(a, b,\eta,\varepsilon, A',B,\alpha,\beta,\gamma,s,t, n)y_n{}^{\frac{2}{ab}},\ \ \ \forall y\in V_0.
	\end{equation}
	Starting from \eqref{eq:compariosn-4}, we provide estimates for various cases when $y$ lies at different locations.
	
	We first consider $y=(0,y_n)$, i.e. $y$ lies on the $x_n$-axis. Since $P$ is of  interior-$(a, \eta, \varepsilon)$-type  with $a\in [2,+\infty)$, then by Lemma \ref{lemma:interior}, $P$ satisfies the interior sphere condition. Thus there exists a constant $h>0$ such that 
	\[y_n=\operatorname{dist}(y,P),\ \ \  \forall y=(0,y_n)\in V_0\cap V_1,\]
	where 
	\begin{equation*}
		V_1:=\left\{x\in\mathbb{R}^n:\ x_n<h\right\}.
	\end{equation*}
	Now it follows from \eqref{eq:compariosn-4} that
	\begin{equation}\label{eq:case1}
		|u(y)-u(P)|  \geq
		C(a, b,\eta,\varepsilon, A,B,\alpha,\beta,\gamma,s,t, n)\left(\operatorname{dist}(y,P)\right)^{\frac{2}{ab}}, \ \ \ \forall y=(0,y_n)\in V_0\cap V_1.
	\end{equation}

	Next we consider points in $V_0\cap V_1$ but not on the $x_{n}$-axis. For any point $y_0\in V_0\cap V_1$ that is not on the $x_{n}$-axis, we denote   $\theta(y_0)$ as the angle between $l_{y_0P}$ and the $x_{n}$-axis. 
	Then for any $y\in V_0\cap V_1\cap l_{y_0P}$, using 
	$y_{n} = d_{y}\cos\theta(y_0)= \operatorname{dist}(y,P)\cos\theta(y_0)$, 
	we can derive from  \eqref{eq:compariosn-4} that
	\begin{align*}
		|u(y)-u(P)| 
		&\geq   C(a, b,\eta,\varepsilon, A,B,\alpha,\beta,\gamma,s,t, n) \big(\operatorname{dist}(y,P)\cos\left(\theta(y_0)\right)\big){}^{\frac{2}{ab}} \\ 
		&=  C(a, b,\eta,\varepsilon, A,B,\alpha,\beta,\gamma,s,t, n)  \left(\operatorname{dist}(y,P)\right)^{\frac{2}{ab}},\stepcounter{equation}\tag{\theequation}\label{eq:case2}
	\end{align*}
	where in the last step we have used the fact that 
	$\big(\cos\left(\theta(y_0)\right)\big){}^{\mu(a)}=C(a,\eta,\varepsilon,\alpha,\beta,\gamma,s,t,n)$. Since $y_0$ is arbitrarily fixed, we note that \eqref{eq:case2} actually holds for any $y\in V_0\cap V_1$ that is not on the $x_n$-axis. Combined with \eqref{eq:case1}, this leads us to 
	\begin{equation*}
		|u(y)-u(P)| 
		\geq  C(a, b,\eta,\varepsilon, A,B,\alpha,\beta,\gamma,s,t, n)  \left(\operatorname{dist}(y,P)\right)^{\frac{2}{ab}},\ \ \ \forall y \in V_0\cap V_1.
	\end{equation*}
	
	We are now in a position to consider $y\in \varOmega_{1/2,P}\setminus (V_0\cap V_1)$. 
	Obviously, there exists at least one  point in $V_0\cap V_1\cap l_{yP}$. Let us fix $y_1\in V_0\cap V_1\cap l_{yP}$. By virtue of \eqref{eq:case2}, we have 
	\begin{equation*}
		|u(y_1)-u(P)| 
		\geq C(a, b,\eta,\varepsilon, A',B,\alpha,\beta,\gamma,s,t, n)  \left(\operatorname{dist}(y_1,P)\right)^{\frac{2}{ab}}.
	\end{equation*}
	It is clear that $u(y)\leq u(y_1)\leq 0$. Hence, by first choosing some constant $0<C(a,\eta,\varepsilon)\leq\frac{\operatorname{dist}(y_1,P)}{\frac{1}{4}\eta\varepsilon^a}$ and then using  $\frac{1}{4}\eta\varepsilon^a\geq \operatorname{dist}(y,P)$ when $y\in\varOmega_{1/2,P}$,  we get 
	\begin{align*}
		|u(y)-u(P)|&\geq |u(y_1)-u(P)|\\
		&\geq C(a, b,\eta,\varepsilon, A,B,\alpha,\beta,\gamma,s,t, n)  \left(\operatorname{dist}(y_1,P)\right)^{\frac{2}{ab}}\\
		&\geq C(a, b,\eta,\varepsilon, A,B,\alpha,\beta,\gamma,s,t, n)  \left(\frac{1}{4}\eta\varepsilon^a\right)^{\frac{2}{ab}}\\
		&\geq C(a, b,\eta,\varepsilon, A,B,\alpha,\beta,\gamma,s,t, n)  \left(\operatorname{dist}(y,P)\right)^{\frac{2}{ab}},\ \ \ \forall y \in \varOmega_{1/2,P}\setminus( V_0\cap V_1).
	\end{align*} 
	
	Combining all the above estimates now gives
	\begin{equation*}
		|u(y)-u(P)| 
		\geq  C(a, b,\eta,\varepsilon, A,B,\alpha,\beta,\gamma,s,t, n)  \left(\operatorname{dist}(y,P)\right)^{\frac{2}{ab}},\ \ \ \forall y \in \varOmega_{1/2,P}.
	\end{equation*}
	We can further obtain that 
	for any $x\in \partial\varOmega$ that is of  interior-$(a, \eta, \varepsilon)$-type , we have 
	\begin{equation*}
		|u(y)-u(x)| 
		\geq  C(a, b,\eta,\varepsilon, A,B,\alpha,\beta,\gamma,s,t, n)  \left(\operatorname{dist}(y,x)\right)^{\frac{2}{ab}},\ \ \ \forall x\in\partial\varOmega,\ \forall y \in \varOmega_{1/2,x}.
	\end{equation*}
	This completes the proof of the lower bound part in Theorem \ref{thm3}.
	
	In view of the fact that 
	\[\frac{2}{ab}\geq  \frac{\beta-\gamma+2t-n-1}{\alpha-\gamma+s+t}+\frac{2 s}{a(\alpha-\gamma+s+t)}=\mu(a),\]
	we can conclude that 
	\begin{equation*}
		|u(y)-u(x)| 
		\geq  C(a, b,\eta,\varepsilon, A,B,\alpha,\beta,\gamma,s,t, n)  \left(\operatorname{dist}(y,x)\right)^{\mu(a)},\ \ \ \forall x\in\partial\varOmega,\ \forall y \in \varOmega_{1/2,x},
	\end{equation*}
	and therefore
	\begin{align*}
		|u(y)-u(x)| 
		\geq  C(a_{\min}, b,\eta,\varepsilon, A,B,\alpha,\beta,\gamma,s,t, n)  \left(\operatorname{dist}(y,x)\right)^{\mu(a_{\min})},\ \ \ \forall x\in\partial\varOmega,\ \forall y \in \varOmega_{1/2,x}.
	\end{align*}
	This also implies 
	\[\text{$\varOmega$ is a lower-$(\mu(a_{\min}),C(a_{\min}, b,\eta,\varepsilon, A,B,\alpha,\beta,\gamma,s,t, n)  )$ type domain about  $u$.}\]
	We have thus proved  case $a\in[2,+\infty)$ for Theorem \ref{thm1-2} and the lower bound part in Theorem \ref{thm1}. Up to now, the proof of  case $a\in[2,+\infty)$ in Theorem \ref{thm1} has been completed.

	\section{Further discussion}\label{sec:discussion}

	\subsection{Convexity parameter $a$ for domain}
	
	In this subsection, we analyze values of the  convexity parameter $a$ for $\varOmega$ and its subdomains, and hence the discussion is restricted to the   case $a\in[2,+\infty)$; we also give an alternative proof of \eqref{eq09} in the case   $a\in[2,+\infty)$ starting from the estimate
	\begin{equation}\label{eq88}
		m\left(\operatorname{dist}(y,P)\right)^{\mu(a_{\varOmega_{1/2,P}}(P))} \leq \vert u(y)-u(P) \vert \leq M\left(\operatorname{dist}(y,P)\right)^{\mu(a_{\varOmega_{1/2,P}}(P))},\quad \forall y\in \varOmega_{1/2,P},
	\end{equation}
	which is collected from  the proof in Section \ref{sec:sub} and Section \ref{sec:super}.
	
	Although $a_{\varOmega_{1/2,P}}(P)\leq a_{\varOmega}(P)$, we note that  $\mu(a)$ is inversely proportional to $a$ by definition, and hence
	\[\mu(a_{\varOmega}(P))\leq \mu(a_{\varOmega_{1/2,P}}(P)).\]
	This means that $\mu(a_{\varOmega}(P))$ will be a  better upper exponent. 
	In fact, we have already used $\mu(a_{\varOmega}(P))$ in Section \ref{sec:sub} based on the fact 
	\begin{equation}\label{eq:a-P}
		a_{\varOmega_{1/2,P}}(P)=a_{\varOmega}(P).
	\end{equation}
	Our target now is to prove  \eqref{eq:a-P}.

	Because $\varOmega_{1/2,P} \subseteq \varOmega$ and $\varOmega_{1/2,P} \cap \varOmega = P$,  it is obvious that
	\begin{equation}\label{eq6.3}
		a_{\varOmega_{1/2,P}}(P) \leq  a_{\varOmega}(P).
	\end{equation}
	Therefore, to prove  (\ref{eq:a-P}), we just need to prove
	\begin{equation}\label{eq6.3'}
		a_{\varOmega_{1/2,P}}(P)\geq a_{\varOmega}(P).
	\end{equation}
	In the following part, our main idea is to verify that there exists a domain $D$ such that 
	\begin{equation}\label{eq6.4}
		D \subset \varOmega_{1/2,P} \subset \varOmega,
		\ \ 	\partial\varOmega_{1/2,P} \cap\partial D = P,\ \ 
		a_{\varOmega_{1/2,P}}(P) =  a_{D}(P).
	\end{equation}

	Here we present a series of definitions and propositions that we need.

	\begin{definition}
		Let $\varOmega$ be a bounded convex domain, $D\subseteq \varOmega$ be convex domain, and $\partial D\cap\partial \varOmega = P$. 
		For any point $Q\in\partial\varOmega$ with $Q\neq P$, 	 
		we define the domain ratio $k_{PQ}(D,\varOmega)$ about $D$ and $\varOmega$ as follows: $$k_{PQ}(D,\varOmega) =\frac{\operatorname{dist}(P,R)}{\operatorname{dist}(P,Q)}>0, 
		$$
		where $R\in \partial D\cap l_{PQ} $ and $R\neq P$.
	\end{definition}

	\begin{proposition}\label{theA6}
		Let $\varOmega$ be a bounded convex domain and $u$ be a convex function over $\varOmega$. If $\varOmega$ is both an upper-$(\mu,M)$-type domain and a lower-$(\nu,m)$-type domain with $0<\nu\leq\mu\leq 1$ and $M\geq m>0$, 
		then there exists some  constants $0<c_1<c_2$, depending only on $\mu$, $\nu$, $M$, $m$ and $\operatorname{diam}(\varOmega)$, such that 
		\begin{align*}
			&\mathop{\inf} _{P,Q\in\partial\varOmega,Q\neq P}{k_{PQ}(\varOmega_{1/2,P},\varOmega)} \geq   c_1,\\
			&\mathop{\sup} _{P,Q\in\partial\varOmega,Q\neq P}{k_{PQ}(\varOmega_{1/2,P},\varOmega)} \leq  c_2. 
		\end{align*}
	\end{proposition} 
	\begin{proof}
		Based on Definition \ref{def2}, 
		for any points $P,Q\in\partial\varOmega$ ($Q\neq P$) and 
		$R\in\partial\varOmega_{1/2,P}\cap\partial\varOmega_{1/2,Q}$, we have 
		\begin{equation*} 
			m\left(\operatorname{dist}(R,P) \right)^{\nu} \leq \vert u(R)-u(P) \vert\leq M\left(\operatorname{dist}(R,P)\right)^{\mu},
		\end{equation*}
		\begin{equation*} 
			m\left(\operatorname{dist}(R,Q) \right)^{\nu} \leq \vert u(R)-u(Q) \vert\leq M\left(\operatorname{dist}(R,Q)\right)^{\mu}.
		\end{equation*}
		We note that 
		\begin{align*}
			k_{PQ}(\varOmega_{1/2,P},\varOmega)
			&=\frac{\operatorname{dist}(P,R)}{\operatorname{dist}(P,Q)}
			= \frac{\operatorname{dist}(P,R)}{\operatorname{dist}(P,R)+\operatorname{dist}(R,Q)}=\frac{1}{1+ \frac{\operatorname{dist}(R,Q)}{\operatorname{dist}(R,P)}}.
		\stepcounter{equation}\tag{\theequation}\label{eq:inf}
		\end{align*}
		Therefore it suffices to deduce 
		\[\mathop{\sup} _{P,Q\in\partial\varOmega,Q\neq P}\frac{\operatorname{dist}(R,Q)}{\operatorname{dist}(R,P)}\ \ \text{ and }\ \ \mathop{\inf} _{P,Q\in\partial\varOmega,Q\neq P}\frac{\operatorname{dist}(R,Q)}{\operatorname{dist}(R,P)},\]
		which is clear to be taken when 
		\begin{equation*}
			 M\left(\operatorname{dist}(R,P)\right)^{\mu}=m\left(\operatorname{dist}(R,Q) \right)^{\nu}\ \ \text{ and }\ \ m\left(\operatorname{dist}(R,P)\right)^{\nu}=M\left(\operatorname{dist}(R,Q) \right)^{\mu}
		\end{equation*}
		respectively, see Figure \ref{Fig.4}. 
		
			\begin{figure}[htbp]  
			\centering 
			\includegraphics[scale=0.45]{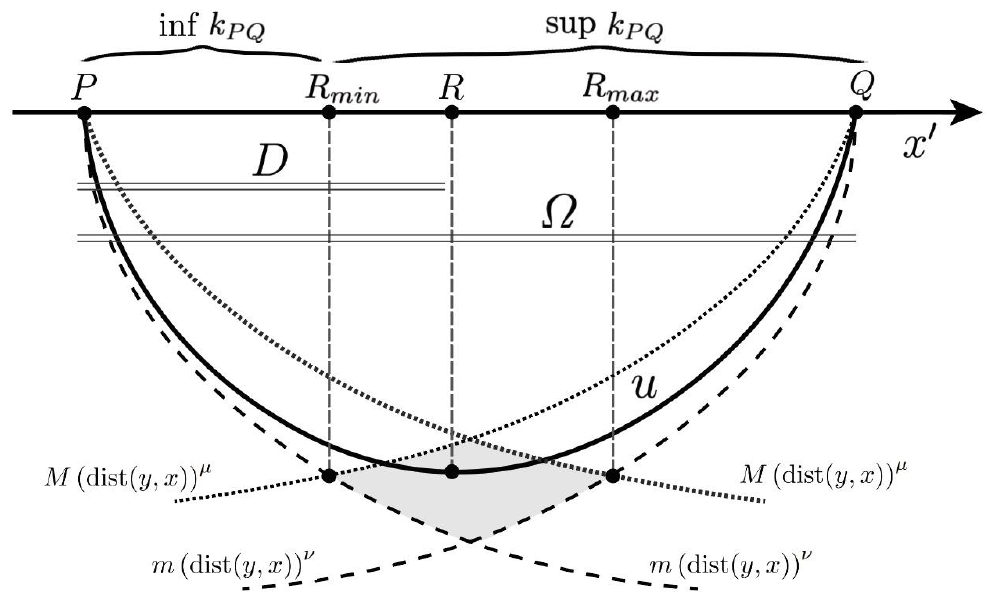} 
			\caption{$k_{PQ}$ touches minimum or maximum value at the point $R_{\min}$ or $R_{\max}$ respectively.}
			\label{Fig.4} 
		\end{figure}
		
		We first consider  $M\left(\operatorname{dist}(R,P)\right)^{\mu}=m\left(\operatorname{dist}(R,Q) \right)^{\nu}$. There are two cases.
		
		\textit{Case 1}:  $\operatorname{dist}(R,Q)\geq 1$. In this case, it follows from  the fact $1-\frac{\nu}{\mu}\geq 0$ that 
		\[\frac{\operatorname{dist}(R,Q) }{\operatorname{dist}(R,P)}= \left(\frac{M}{m}\right)^{\frac{1}{\mu}}\left(\operatorname{dist}(R,Q)\right)^{1-\frac{\nu}{\mu}}  \leq \left(\frac{M}{m}\right)^{\frac{1}{\mu}}\left(\operatorname{diam}(\varOmega)\right)^{1-\frac{\nu}{\mu}}.\]
		
		\textit{Case 2}:  $\operatorname{dist}(R,Q)\leq 1$. Since $\nu\leq\mu$, we can derive
		\begin{align*}
			M\left(\operatorname{dist}(R,P)\right)^{\mu}=m\left(\operatorname{dist}(R,Q) \right)^{\nu}\geq m\left(\operatorname{dist}(R,Q) \right)^{\mu},
		\end{align*}
		which implies that 
		\[\frac{\operatorname{dist}(R,Q) }{\operatorname{dist}(R,P)}\leq \left(\frac{M}{m}\right)^{\frac{1}{\mu}}.\]
		
		To sum up, we have
		\[\mathop{\sup} _{P,Q\in\partial\varOmega,Q\neq P}\frac{\operatorname{dist}(R,Q) }{\operatorname{dist}(R,P)}\leq \left(1+\left(\operatorname{diam}(\varOmega)\right)^{1-\frac{\nu}{\mu}}\right)\left(\frac{M}{m}\right)^{\frac{1}{\mu}}.\]
		Taking this back to \eqref{eq:inf} gives
		\begin{equation*}
			\mathop{\inf} _{P,Q\in\partial\varOmega,Q\neq P}{k_{PQ}(\varOmega_{1/2,P},\varOmega)}\geq \frac{1}{1+ (1+\left(\operatorname{diam}(\varOmega)\right)^{1-\frac{\nu}{\mu}}) \left(\frac{M}{m}\right)^{\frac{1}{\mu}}} =c_1.
		\end{equation*}

		Now we consider  $m\left(\operatorname{dist}(R,P)\right)^{\nu}=M\left(\operatorname{dist}(R,Q) \right)^{\mu}$. There are also two cases.

			\textit{Case 1}:  $\operatorname{dist}(R,Q)\geq 1$. In this case, it follows from   the fact $1-\frac{\nu}{\mu}\geq 0$ that 
			\[\frac{\operatorname{dist}(R,Q) }{\operatorname{dist}(R,P)}= \left(\frac{m}{M}\right)^{\frac{1}{\mu}}\left(\operatorname{dist}(R,Q)\right)^{1-\frac{\nu}{\mu}}  \geq \left(\frac{m}{M}\right)^{\frac{1}{\mu}}.\]

			\textit{Case 2}:  $\operatorname{dist}(R,Q)\leq 1$. Since $\mu\geq\nu$, we can derive
			\begin{align*}
				m\left(\operatorname{dist}(R,P)\right)^{\nu}=M\left(\operatorname{dist}(R,Q) \right)^{\mu}\leq M\left(\operatorname{dist}(R,Q) \right)^{\nu},
			\end{align*}
			which implies that 
			\[\frac{\operatorname{dist}(R,Q) }{\operatorname{dist}(R,P)}\geq \left(\frac{m}{M}\right)^{\frac{1}{\nu}}.\]

			In conclusion, we obtain
			\[\mathop{\inf} _{P,Q\in\partial\varOmega,Q\neq P}\frac{\operatorname{dist}(R,Q) }{\operatorname{dist}(R,P)}\geq\min\left\{\left(\frac{m}{M}\right)^{\frac{1}{\mu}},\left(\frac{m}{M}\right)^{\frac{1}{\nu}}\right\}=\left(\frac{m}{M}\right)^{\frac{1}{\nu}}.\]
			Taking this back to \eqref{eq:inf} finally yields
			\begin{equation*}
				\mathop{\sup} _{P,Q\in\partial\varOmega,Q\neq P}{k_{PQ}(\varOmega_{1/2,P},\varOmega)}\leq \frac{1}{1+  \left(\frac{m}{M}\right)^{\frac{1}{\nu}}}=c_2,
		\end{equation*}
		proving the proposition.
	\end{proof}

	\begin{proposition}\label{pro1}
		Two homothetic domains that are tangent to each other have the same value of  $a$ for the exterior-$(a,\eta)$-type pairs at their tangent point.
	\end{proposition}
	Intuitively, 
	Proposition \ref{pro1} is clear from  Figure \ref{Fig.5} and we omit its proof. Based on Proposition \ref{theA6} and Proposition \ref{pro1}, there hold  the following direct consequences.

	\begin{corollary}\label{coro}
		Assume all the conditions in Proposition \ref{theA6} hold. Then we have
		\begin{enumerate}[(i)]
			\item there exists a domain $D\subseteq \varOmega$ such that $\partial\varOmega\cap\partial D = P$, $D\subseteq \varOmega_{1/2,P}$,  and  $D$ is homothetic with $\varOmega$;
			\item 	the values of $a$ in the exterior-$(a,\eta)$-type  pairs for both $\varOmega$ and $\varOmega_{1/2,P}$ at   the point $P$ are equal (i.e. $a=a_{\min}$);
			\item the values of $a$ in the interior-$(a,\eta,\varepsilon)$-type pairs for  both $\varOmega$ and $\varOmega_{1/2,P}$ at the point $P$ are equal (i.e. $a=a_{\max}$).
		\end{enumerate}
	\end{corollary}
	
\begin{figure}[htbp]  
	\centering 
	\includegraphics[scale=0.45]{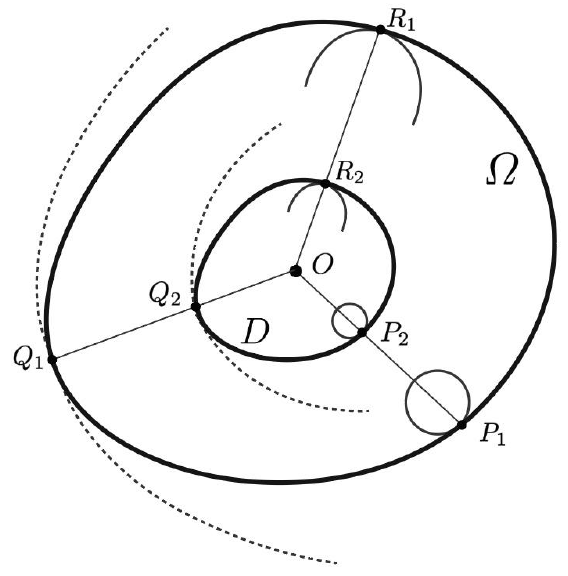} 
	\caption{Two homothetic domains that have the same $a$ at pair point.} 
	\label{Fig.5} 
\end{figure}
	
	Because of Proposition \ref{theA6} and Corollary \ref{coro} (ii), we are ready to construct a domain $D$ that satisfies  (\ref{eq6.4}).
	Since $D \subset \varOmega_{1/2,P}$,  we have
	\begin{equation*}
		a_{D}(P)   \leq a_{\varOmega _{1/2,P}}(P).
	\end{equation*}
	From Proposition \ref{pro1}, we know that 
	\begin{equation*}
		a_{D}(P) = a_{\varOmega}(P). 
	\end{equation*}
	Therefore we obtain \eqref{eq6.3'}. Together with 
	(\ref{eq6.3}), this gives the desired result \eqref{eq:a-P}.

	Taking \eqref{eq:a-P}  back to   (\ref{eq88}) now leads us to
	\begin{equation*}
		m\left(\operatorname{dist}(y,P)\right)^{\mu(a_{\varOmega}(P))} \leq \vert u(y)-u(P) \vert \leq M\left(\operatorname{dist}(y,P)\right)^{\mu(a_{\varOmega}(P))} ,\quad  \forall y\in \varOmega_{1/2,P}.
	\end{equation*}
	By Lemma \ref{lemma-2},  we can derive the final conclusion that
	\begin{equation*}
		\vert u(y)-u(P) \vert\ = \mathcal{O} (\left(\operatorname{dist}(y,P)\right)^{\mu(a_{\varOmega}(P))}),\quad  \forall y\in \varOmega_{1/2,P},
	\end{equation*} 
	which is consistent with our result \eqref{eq09} in Theorem \ref{thm1}.
	
\subsection{Rest proof: The case $a\in [1,2)$}
	
	In this case, we can only construct local estimates (since there is no exterior-$(a,\eta)$-type domain with $a\in [1,2)$) and 	
	 we always have $b>1$ (due to the H\"older exponent $\mu(a)=\frac{2}{ab}\leq 1$).  
 For convenience, all  constants $\widetilde{C}$s in this case different from those $C$s in the case $a\in [2,+\infty)$ are collected in Table \ref{table:constant-3}.

\begin{table}[ht] \caption{Collection of  constants  changing for the case $a\in [1,2)$}
	\label{table:constant-3}
	\renewcommand{\arraystretch}{1.8}\centering
	\begin{tabular}{|c|c|c|c|c|} 
		\hline
		Symbol & \multicolumn{2}{|c|}{Values}  
		& Sign
		\\ [0.5ex]
		\hline
		$\widetilde{C}_1$ 
		& \multicolumn{2}{|c|}{$\frac{8(b-1)}{a^2 b^3} (1-\delta)$}  &$\widetilde{C}_1>0$ \\\hline 
		$\widetilde{C}_2$ 
		& \multicolumn{2}{|c|}{$\left(\frac{2}{b}+\frac{4(b-1)}{b^2} \delta(1-\delta)^{-1}\right) \left(\frac{2(a-2)}{a^2 b}(1-\delta)^{a-1}+\frac{4(b-1)}{a^2 b^2}(1-\delta)^{a-2} \right)$}  &$\widetilde{C}_2>0$ \\\hline
		$\widetilde{C}_3$ 
		& \multicolumn{2}{|c|}{$\left(\frac{2}{b}+\frac{4(b-1)}{b^2} \delta(1-\delta)^{-1}\right) (\operatorname{\operatorname{diam}}(\varOmega))^{2-\frac{2}{a}}+\frac{2(a-2)}{a^2 b}(1-\delta)^{a-1}+\frac{4(b-1)}{a^2 b^2}(1-\delta)^{a-2}$ }  &$\widetilde{C}_3>0$  \\\hline
		$\widetilde{C}_4$ 
		& \multicolumn{2}{|c|}{$\min \left\{\frac{2}{b}, \frac{\widetilde{C}_1(a, b, \delta)}{\widetilde{C}_3(a, b, \delta, \operatorname{\operatorname{diam}}(\varOmega))}\right\}$}  &$\widetilde{C}_4>0$\\\hline
		$\widetilde{C}_5$ 
		& \multicolumn{2}{|c|}{$\frac{2(ab-2)}{a^2b^2}$ }  &$\widetilde{C}_5>0$ 
		\\\hline
		$\widetilde{C}_{6}$ 
		& \multicolumn{2}{|c|}{$4\big(\rho'^{-1}+1\big)^2\big(1+(2a\eta' \varepsilon^{a-1})^2\big) \max\Big\{\Big(\big(1-\big(\frac{1}{4}\big)^{\frac{2}{a}}\big)\Big)^{2-a},1\Big\}$}   &$\widetilde{C}_{6}>0$ \\\hline
		$\widetilde{C}_{7}$ 
		& \multicolumn{2}{|c|}{$\min\left\{2^{-\gamma}a^\gamma b^\gamma,\widetilde{C}_{6}(a,b,\eta',\varepsilon)^{-\frac{\gamma}{2}}\right\}$}   &$\widetilde{C}_{7}>0$ 
		\\\hline
		$\widetilde{C}_{10}$ 
		& \multicolumn{2}{|c|}{\begin{tiny}$A^{-1}B  \widetilde{C}_4(a, b, \delta, \operatorname{\operatorname{diam}}(\varOmega))^s \widetilde{C}_5(a, b)^t  \widetilde{C}_7(a, b,\eta',\varepsilon, \gamma)C_8(a,\eta',\varepsilon,\beta,n)\min\Big\{\big(1-\big(\frac{1}{4}\big)^{\frac{2}{a}}\big)^{\frac{1}{b}\left(\alpha+(b-1) \gamma+(1-b) s+(1-ab) t\right)},1\Big\}$\end{tiny}}   &$\widetilde{C}_{10}>0$  \\\hline
		$\widetilde{C}_{12}$&\multicolumn{2}{|c|}{ $\widetilde{C}_{10}(a, b,\eta',\varepsilon, A,B,\alpha,\beta,\gamma,s,t, \delta,      \operatorname{\operatorname{diam}}(\varOmega), n)C_{11}(a,b,\alpha,\beta,\gamma,s,t,\operatorname{diam}(\varOmega),n) $ }   &$ 
		\widetilde{C}_{12}>0$  \\\hline
		$\widetilde{C}_{16}$ 
		& \multicolumn{2}{|c|}{$\frac{16}{a^2}\big(\rho^{-1}+1\big)^2\big(1+(2a\eta \varepsilon^{a-1})^2\big)$}   &$\widetilde{C}_{16}>0$  \\\hline
		$\widetilde{C}_{17}$ 
		& \multicolumn{2}{|c|}{$\max\left\{2^{-\gamma}a^\gamma b^\gamma,\widetilde{C}_{16}(a,b,\eta,\varepsilon)^{-\frac{\gamma}{2}}\right\}$}   &$\widetilde{C}_{17}>0$  \\\hline
		$\widetilde{C}_{18}$& \multicolumn{2}{|c|}{$A^{-1}BC_{13}(b)^{s}C_{14}(a,b,\eta,\varepsilon,n)^{t}C_{15}(a,\eta,\varepsilon,\beta,n)\widetilde{C}_{17}(a,b,\eta,\varepsilon)$  } &$\widetilde{C}_{18}>0$  \\\hline
		$\widetilde{C}_{20}$& \multicolumn{2}{|c|}{$\widetilde{C}_{18}(a, b,\eta,\varepsilon, A,B,\beta,\gamma,s,t, n)C_{19}(a,b,\alpha,\gamma,s,t) $}   &$C_{20}>0$    \\\hline
		$\widetilde{C}_{22}$& \multicolumn{2}{|c|}{$\widetilde{C}_{20}(a, b,\eta,\varepsilon, A,B,\alpha,\beta,\gamma,s,t, n)  C_{21}(a,b,\eta,\varepsilon,\alpha,\beta,\gamma,s,t,n)$}   &$\widetilde{C}_{22}>0$  \\\hline
	\end{tabular}
\end{table}	 
	 
	 Though several constants   for this case  become different, we run the same  procedure in the last two sections here  except for the gradient term $(1 + |DW|^2)^{-\frac{\gamma}{2}}$.

In the proof of boundary upper bound estimates for this case, in order to get a lower bound for the gradient term $(1 + |DW|^2)^{-\frac{\gamma}{2}}$ locally, we first use  Remark \ref{remark:eta'} to get the domain  $V'\subseteq \varOmega_{1/2,P}$  as defined by \eqref{def:V'};  then similar to \eqref{eq:W-bound} and \eqref{eq:DW-1}, it holds in $V'$ that 
\begin{equation*}
	|W|\in\left[\big(1-\big(\frac{1}{4}\big)^{\frac{2}{a}}\big)^{\frac{1}{b}}\cdot\big(\frac{x_n}{\xi}\big)^{\frac{2}{ab}},\big(\frac{x_n}{\xi}\big)^{\frac{2}{ab}}\right],
\end{equation*}
and
\begin{equation*}
	1+|DW|^2\in \left[\frac{4}{a^2b^2}\cdot|W|^{2-2b}\cdot x_n{}^{\frac{4}{a}-2}\cdot\xi^{-\frac{4}{a}},\widetilde{C}_{6}(a,b,\eta',\varepsilon)\cdot|W|^{2-2b}\cdot x_n{}^{\frac{4}{a}-2}\cdot\xi^{-\frac{4}{a}}\right]
\end{equation*}
with the constant  $\rho'=\rho'(a, \eta', \varepsilon)>0$ such that  
$|W|\geq \rho'\operatorname{dist}(x,S')$ for any $x\in V'$; and finally we derive
\begin{equation*}
	\left(1+|DW|^2\right)^{-\frac{\gamma}{2}} \geq \widetilde{C}_{7}(a,b,\eta',\varepsilon)  \cdot|W|^{(b-1) \gamma}\cdot x_n{}^{(1-\frac{2}{a})\gamma}\cdot \xi^{\frac{2}{a}\gamma}.
\end{equation*} 
The steps concerning $\widetilde{C}_{8}$ and $\widetilde{C}_{9}$ can be skipped now.

In the proof of boundary lower bound estimates for this case,
 we also refer to the geometric handling in \cite{Li-Li-F} to bound  $1+|D W|^2$ for any $x\in V$: 
\begin{equation*} 
	1+|DW|^2\in \left[\frac{4}{a^2b^2}\cdot|W|^{2-2b}\cdot x_n{}^{\frac{4}{a}-2}\cdot\xi^{-\frac{4}{a}},\widetilde{C}_{16}(a,b,\eta,\varepsilon)\cdot|W|^{2-2b}\cdot x_n{}^{\frac{4}{a}-2}\cdot\xi^{-\frac{4}{a}}\right].
\end{equation*}
In this way, we get the following upper bound for the gradient term $(1 + |DW|^2)^{-\frac{\gamma}{2}}$ locally: 
\begin{equation*}
	\left(1+|DW|^2\right)^{-\frac{\gamma}{2}} \leq \widetilde{C}_{17}(a,b,\eta,\varepsilon)  \cdot|W|^{(b-1) \gamma}\cdot x_n{}^{(1-\frac{2}{a})\gamma}\cdot \xi^{\frac{2}{a}\gamma}.
\end{equation*}

	\subsection{Application and example}\label{sec:apply}
	
	There are fruitful applications of our results. Here we merely apply  our results to the scenarios that were  considered in \cite{Li-Li-F,Li-Li-N,Li-2021} for example.

	On  one hand, 
	the results in this paper can be applied to the Dirichlet problem of  fully nonlinear ellliptic equations with zero boundary conditions:
	\begin{align*}
		F(\lambda_{1}(D^{2}u),\cdot \cdot \cdot ,\lambda_{n}(D^{2}u)) &= f(x,u,Du)\quad \text{in} \; \varOmega,   \\
		u &= 0\quad \text{on} \; \partial\varOmega. 
	\end{align*}
	If we let $F$ satisfy $(F_1)$-$(F_2)$ and $f$ satisfy $(f_1)$-$(f_3)$, then 
	Theorem \ref{thm1-1} in this paper includes Theorem 1.2 in \cite{Li-Li-N} as a special case when $\varphi\equiv 0$. If we let $F$ satisfy $(F_1)$ and $(F_2')$ and let $f$ satisfy $(f_1)$, $(f_2)$ and $(f_3')$, then 
	Theorem \ref{thm1-2} in this paper includes Theorem 1.1 in \cite{Li-Li-F} as a special case when $\varphi\equiv 0$. If we let $F$ satisfy \eqref{eq:Fcase},  $f$ satisfy \eqref{eq:fcase}, and $\varphi$ satisfy \eqref{eq:varphicase}, then Theorem \ref{thm1} in this paper includes both Theorem 1.2 in \cite{Li-Li-N} and Theorem 1.1 in \cite{Li-Li-F}.  
	
	On the other hand, we can apply the results in this paper to the Dirichlet problem of  Monge-Amp\`ere type equations:
	\begin{align*}
		\det D^2u &= f(x,u,Du)\quad \text{in} \; \varOmega \\
		u &= \varphi (x)\quad \text{on} \; \partial\varOmega.
	\end{align*} 
	If we let  $f$ satisfy $(f_1)$-$(f_3)$ and $\varphi\in C^{\mu}(\partial\varOmega)$, then Theorem \ref{thm1-1} in this paper includes Theorem 1.1 and  Theorem 1.2 in \cite{Li-2021} where $\varphi$ was assumed with more restricted conditions.

	Let us give a more specific example at the end of this paper. 
	We consider the following Dirichlet problem of the hyperbolic affine sphere equation:
	\begin{align}
		\det D^{2}u &= |u|^{-(n+2)} \ \ \text{ in }  B_{\frac{1}{2}}(0,\cdots,0,\tfrac{1}{2}), \stepcounter{equation}\tag{\theequation}\label{eq93}\\ 
		u &= -\sqrt{x_n}\ \ \text{ on }  \partial B_{\frac{1}{2}}(0,\cdots,0,\tfrac{1}{2}).\stepcounter{equation}\tag{\theequation}\label{eq93-boundary}
	\end{align}
	In this special case, we take
	\begin{align*}
		&\varOmega = B_{\frac{1}{2}}(0,\cdots,0,\tfrac{1}{2}),\ \ \ \ \ \ \ \ 
		F(\lambda_{1}(D^{2}u),\cdots ,\lambda_{n}(D^{2}u)) =\det D^{2}u=\lambda_{1}(D^{2}u)\cdots \lambda_{n}(D^{2}u),\\
		&f(x,u,Du) = |u|^{-(n+2)},\ \ \ \ 
		\varphi(x)=-\sqrt{x_n}.
	\end{align*}
	Here, $\varOmega$ is a $(2,\eta)$ type domain (which is consistent with Lemma \ref{lemma:Jian-Li}); $P=O\in\partial\varOmega$; and we can take $\alpha = n+2$, $\beta = n+1$, $\gamma = 0$ in both $(f_3)$ and $(f_3')$.  
	
	It is easy to verify that $$U(r,x_n):  
	=-\sqrt{1-r^2-(x_{n}-1)^{2}}$$ is a solution to the problem (\ref{eq93})-\eqref{eq93-boundary}.  
	Direct computations give rise to 
	\begin{align*}
		U_r=-\frac{r}{U},\ \ U_n=-\frac{x_n-1}{U},\ \ U_{rr}=-\frac{U^2+r^2}{U^3},\ \ U_{rn}=-\frac{r(x_n-1)}{U^3},\ \ U_{nn}=-\frac{U^2+(x_n-1)^2}{U^3}.
	\end{align*}
	Using Lemma \ref{lemma:eigen}, we can infer that
	\[\lambda_-(D^2U)=-\frac{1}{U^3},\ \ \lambda_+(D^2U)=-\frac{1}{U},\]
	and hence the $n$ eigenvalues of $D^2U$ are
	\[\underbrace{-\frac{1}{U}, \cdots, -\frac{1}{U}}_{n-1\ \text{terms}}, -\frac{1}{U^3}.\]
	It follows that 
	\begin{equation*}
		F(\lambda_{1}(D^{2}u),\cdots ,\lambda_{n}(D^{2}u)) 
		=\Big(-\frac{1}{U}\Big)^{n-1}\Big(-\frac{1}{U^3}\Big)=\lambda_{\min}(D^2U)^{n-1}\lambda_{\max}(D^2U),
	\end{equation*}
	where we have used the facts 
	\[\lambda_{\min}(D^2U)=-\frac{1}{U},\ \ \lambda_{\max}(D^2U)=-\frac{1}{U^3}.\]
	Comparing with the structure conditions  shows that we can take $s=n-1$, $t=1$ in both $(F_2)$ and  $(F_2')$.  
	
	By direct computations, 
	we obtain $b=2$ through \eqref{eq:take-b-1} while $b\leq 2$ through \eqref{eq:take-b-2}. Thus we take $b=2$ and calculate the H\"older exponent as $\mu(a)=\frac{1}{2}$. We are now ready to use  Theorem \ref{thm1} to derive that  for any    $y\in \varOmega_{1/2,P}$ that 
	\begin{equation}\label{eq:verify}
		\vert U(y)-U(P) \vert=\mathcal{O}\left(\left(\operatorname{dist}(y,P)\right)^{\frac{1}{2}}\right)
	\end{equation}
and 
	$$u\in C^{\frac{1}{2}}(\overline{B_{\frac{1}{2}}(0,\cdots,0,\tfrac{1}{2})} ).$$ 
	
	In fact,  this result can by verified as follows. 
	We note that  there always exists some vector  
	$
	\kappa= (\kappa_1,\cdots,\kappa_n)$ with $\kappa_n>0$ such that 
	$y= P+t\kappa$, where $t$ is used as the parameter. It is straightforward to calculate that
	\begin{align*}
		\frac{|U(y)-U(P)|}{\left(\operatorname{dist}(y,P)\right)^{\frac{1}{2}}}=\frac{\big|\sqrt{1-(t\kappa_1)^2-\cdots (t\kappa_{n-1})^2-(t\kappa_n-1)^2}-0\big|}{\sqrt{t|\kappa|}}=\frac{\sqrt{2\kappa_n-|\kappa|^2t } }{\sqrt{|\kappa|}}.
	\end{align*}
	In this way, we obtain 
	\begin{equation*}
		\lim_{t\to 0}\frac{|U(y)-U(P)|}{\left(\operatorname{dist}(y,P)\right)^{\frac{1}{2}}}=\frac{\sqrt{2\kappa_n} }{\sqrt{|\kappa|}},
	\end{equation*}
	which is a positive constant. This proves \eqref{eq:verify} and therefore $\mu(a)=\frac{1}{2}$ is the optimal exponent.  We also note that 
	\[1-r^2-(x_{n}-1)^{2}\neq 0\ \  \text{over}\  \overline{B_{\frac{1}{2}}(0,\cdots,0,\tfrac{1}{2})},\]
which implies that 
	\begin{equation*} 
		U(r,x_n) 
		=-\sqrt{1-r^2-(x_{n}-1)^{2}} 
	\end{equation*}
is smooth over $\overline{B_{\frac{1}{2}}(0,\cdots,0,\tfrac{1}{2})}$. This gives 
		$$u\in C^{\frac{1}{2}}(\overline{B_{\frac{1}{2}}(0,\cdots,0,\tfrac{1}{2})} )$$ 
		as a direct consequence.

	 \appendix
\section{Some computations on barrier function $W$}\label{appendix}
 
For convenience, we gather some  computations related to the constants $C_1$, $C_2$, $C_3$, $\widetilde{C}_1$, $\widetilde{C}_2$ and $\widetilde{C}_3$, as  they are not easy to obtain.

By \eqref{eq:W} and \eqref{assumption-7}, 
we  infer the  relations  
\begin{align*}
	&r^2\in\Big[0,\delta\big(\frac{x_{n}}{\xi}\big)^{\frac{2}{a}}\Big]\subseteq\Big[0,\delta(1-\delta)^{-1}|W|^b\Big] ,\stepcounter{equation}\tag{\theequation}\label{boundofr2}\\
	&|W|^b\in \Big[(1-\delta)\big(\frac{x_{n}}{\xi}\big)^{\frac{2}{a}},\big(\frac{x_{n}}{\xi}\big)^{\frac{2}{a}}\Big],\stepcounter{equation}\tag{\theequation}\label{boundofW}\\
	&\frac{x_n}{\xi} \in \Big[|W|^{\frac{ab}{2}},(1-\delta)^{-\frac{a}{2}}|W|^{\frac{ab}{2}}\Big],\stepcounter{equation}\tag{\theequation}\label{boundofxneps}
\end{align*}
and also obtain
\begin{align*}
	W_{rr}\cdot W_{nn}-|W_{rn}|^2&=\underbrace{\frac{8(a-2)(b-1)}{a^2b^3}\cdot|W|^{2-3b}\cdot\big(\frac{x_n}{\xi}\big)^{\frac{2}{a}-2}\cdot r^2\cdot\xi^{-2}}_{I_1}\\
	&\ \ +\underbrace{\frac{8(b-1)}{a^2b^3}\cdot|W|^{2-3b}\cdot\big(\frac{x_n}{\xi}\big)^{\frac{4}{a}-2}\cdot\xi^{-2}}_{I_2}
	+\underbrace{\frac{4(a-2)}{a^2b^2}\cdot|W|^{2-2b}\cdot\big(\frac{x_n}{\xi}\big)^{\frac{2}{a}-2}\cdot\xi^{-2}}_{I_3}.
	\stepcounter{equation}\tag{\theequation}\label{I123}
\end{align*}

For the case $a\in [2,+\infty)$, the estimates in \textit{Step 1-1}
and \textit{Step 1-2} are different.

In \textit{Step 1-1}, we have $a\geq 2$ and $b>1$. 
It is easy to see from \eqref{I123} that  $I_1$, $I_2$, $I_3\geq  0$ and then \eqref{boundofxneps} gives 
\begin{align*}
	W_{rr}\cdot W_{nn}-|W_{rn}|^2
	&\geq  I_2=\frac{8(b-1)}{a^2b^3}\cdot|W|^{2-3b}\cdot\big(\frac{x_n}{\xi}\big)^{\frac{4}{a}-2}\cdot\xi^{-2}\\
	&\geq \frac{8(b-1)}{a^2b^3}\cdot|W|^{2-3b}\cdot(1-\delta)^{a-2}|W|^{2b-ab}\cdot\xi^{-2}\\
	&=	C_1(a,b,\delta)\cdot|W|^{2-b-ab}\cdot\xi^{-2}.
\end{align*}
By virtue of \eqref{boundofr2}, it is trivial to infer  
\begin{equation*}
	W_{rr}\in \left[\frac{2}{b}\cdot|W|^{1-b},\left(\frac{2}{b}+\frac{4(b-1)}{b^2}\delta(1-\delta)^{-1}\right)\cdot|W|^{1-b}\right].
\end{equation*}
According to   \eqref{boundofxneps}, we get 
\begin{equation*}
	W_{nn}\in \left[\left(\frac{2(a-2)}{a^2 b}(1-\delta)^{a-1}+\frac{4(b-1)}{a^2 b^2}(1-\delta)^{a-2} \right)\cdot|W|^{1-ab}\cdot\xi^{-2},\left(\frac{2(a-2)}{a^2 b}+\frac{4(b-1)}{a^2 b^2}\right)\cdot|W|^{1-ab}\cdot\xi^{-2}\right].
\end{equation*}
Hence we obtain 
\begin{align*}
	W_{rr}\cdot W_{nn}-|W_{rn}|^2
	&\leq  W_{rr}\cdot W_{nn}\leq  C_2(a,b,\delta)\cdot |W|^{2-b-ab}\cdot\xi^{-2}.
\end{align*}
We can also derive that 
\begin{align*}
	W_{rr}+W_{nn}&\leq   \left(\frac{2}{b}+\frac{4(b-1)}{b^2}\delta(1-\delta)^{-1}\right)\cdot |W|^{1-b}+\left(\frac{2(a-2)}{a^2 b}+\frac{4(b-1)}{a^2 b^2}\right)\cdot|W|^{1-ab}\cdot\xi^{-2}\\
	&=\left(\left(\frac{2}{b}+\frac{4(b-1)}{b^2}\delta(1-\delta)^{-1}\right)\cdot |W|^{ab-b}\cdot \xi^2+\frac{2(a-2)}{a^2 b}+\frac{4(b-1)}{a^2 b^2}\right)\cdot|W|^{1-ab}\cdot\xi^{-2}.
\end{align*}
In view of $2-\frac{2}{a}>0$, we note that 
\begin{equation*}
	|W|^{ab-b}\leq \big(\frac{x_{n}}{\xi}\big)^{2-\frac{2}{a}}\leq  \big(\operatorname{diam}(\varOmega)\big)^{2-\frac{2}{a}}\cdot \xi^{\frac{2}{a}-2}.
\end{equation*}
Since $0<\xi<1$, it follows that 
\begin{align*}
	W_{rr}+W_{nn}&\leq   \left(\left(\frac{2}{b}+\frac{4(b-1)}{b^2}\delta(1-\delta)^{-1}\right)\cdot \big(\operatorname{diam}(\varOmega)\big)^{2-\frac{2}{a}}\cdot \xi^{\frac{2}{a}} +\frac{2(a-2)}{a^2 b}+\frac{4(b-1)}{a^2 b^2}\right)\cdot|W|^{1-ab}\cdot\xi^{-2}\\
	&\leq C_3(a, b, \delta, \operatorname{\operatorname{diam}}(\varOmega)) \cdot|W|^{1-a b} \cdot \xi^{-2}.
\end{align*}

In \textit{Step 1-2}, we have $a\geq 2$ and $0<b\leq 1$. 
It follows that $I_1\leq  0$, $I_2\leq  0$, $I_3\geq  0$ in \eqref{I123}. Moreover, using  \eqref{boundofr2}, we get 
\[I_1=(a-2)r^2\cdot\big(\frac{x_n}{\xi}\big)^{-\frac{2}{a}}\cdot I_2\geq \delta(a-2)\cdot I_2.\] 
Hence, it follows from \eqref{boundofxneps} that
\begin{align*}
	I_1+I_2&\geq \big(\delta(a-2)+1\big)\cdot I_2\\
	&=\big(\delta(a-2)+1\big)\frac{8(b-1)}{a^2b^3}\cdot|W|^{2-3b}\cdot\big(\frac{x_n}{\xi}\big)^{\frac{4}{a}-2}\cdot\xi^{-2}\\
	&\geq \big(\delta(a-2)+1\big)\frac{8(b-1)}{a^2b^3}\cdot|W|^{2-3b}\cdot|W|^{2b-ab}\cdot\xi^{-2}\\
	&=\big(\delta(a-2)+1\big)\frac{8(b-1)}{a^2b^3}\cdot|W|^{2-b-ab}\cdot\xi^{-2}.
\end{align*}
We also note that \eqref{boundofxneps} implies
\begin{align*}
	I_3&= \frac{4(a-2)}{a^2b^2}\cdot|W|^{2-2b}\cdot\big(\frac{x_n}{\xi}\big)^{\frac{2}{a}-2}\cdot\xi^{-2}\\
	&\geq 
	\frac{4(a-2)}{a^2b^2}\cdot|W|^{2-2b}\cdot(1-\delta)^{a-1}|W|^{b-ab}\cdot\xi^{-2}\\
	&=\frac{4(a-2)}{a^2b^2}\cdot(1-\delta)^{a-1}\cdot|W|^{2-b-ab}\cdot\xi^{-2}.
\end{align*}
Consequently, we have 
\begin{align*}
	W_{rr}\cdot W_{nn}-|W_{rn}|^2&\geq \delta(a-2)\cdot I_2+I_2+I_3\\ &\geq \big(\delta(a-2)+1\big)\frac{8(b-1)}{a^2b^3}\cdot|W|^{2-b-ab}\cdot\xi^{-2}+\frac{4(a-2)}{a^2b^2}\cdot(1-\delta)^{a-1}\cdot|W|^{2-b-ab}\cdot\xi^{-2}\\
	&=C_1(a,b,\delta)
	\cdot|W|^{2-b-ab}\cdot\xi^{-2}.
\end{align*}
By \eqref{boundofr2}, we obtain 
\begin{equation*}
	W_{rr}\in \left[\left(\frac{2}{b}+\frac{4(b-1)}{b^2}\delta(1-\delta)^{-1}\right)|W|^{1-b},\frac{2}{b}\cdot|W|^{1-b}\right].
\end{equation*} 
According to  \eqref{boundofxneps}, we get 
\begin{equation*}
	W_{nn}\in \left[\left(\frac{2(a-2)}{a^2b}(1-\delta)^{a-1}+\frac{4(b-1)}{a^2b^2}\right)\cdot|W|^{1-ab}\cdot\xi^{-2},\left(\frac{2(a-2)}{a^2b}+\frac{4(b-1)}{a^2b^2}(1-\delta)^{a-2}\right)\cdot|W|^{1-ab}\cdot\xi^{-2}\right],
\end{equation*}
where we can take $\delta=C(b)$ small enough such that the coefficients are positive. Hence we derive
\begin{align*}
	W_{rr}\cdot W_{nn}-|W_{rn}|^2
	&\leq  W_{rr}\cdot W_{nn}\leq  C_2(a,b,\delta)\cdot |W|^{2-b-ab}\cdot\xi^{-2},
\end{align*}
and 
\begin{align*}
	W_{rr}+W_{nn}&\leq   \frac{2}{b}\cdot|W|^{1-b}+\left(\frac{2(a-2)}{a^2b}+\frac{4(b-1)}{a^2b^2}(1-\delta)^{a-2}\right)\cdot|W|^{1-ab}\cdot\xi^{-2}\\
	&=\left(\frac{2}{b}\cdot |W|^{ab-b}\cdot \xi^2+\left(\frac{2(a-2)}{a^2b}+\frac{4(b-1)}{a^2b^2}(1-\delta)^{a-2}\right)\right)\cdot|W|^{1-ab}\cdot\xi^{-2}\\
	&\leq   \left(\frac{2}{b}\cdot \left(\operatorname{diam}(\varOmega)\right)^{2-\frac{2}{a}}\cdot \xi^{\frac{2}{a}} +\left(\frac{2(a-2)}{a^2b}+\frac{4(b-1)}{a^2b^2}(1-\delta)^{a-2}\right)\right)\cdot|W|^{1-ab}\cdot\xi^{-2}\\
	&\leq  C_3(a,b,\delta,\operatorname{diam}(\varOmega))\cdot |W|^{1-ab}\cdot\xi^{-2}.
\end{align*}

	For the case $a\in [1,2)$, we always have $a<2$,  $b>1$ and $ab > 2$.
	It is obvious that $I_1<0$, $I_2>0$ and $I_3<0$. Since we let $\xi$ sufficiently small when constructing sub-solutions, we note that for each   $x_n$, $I_2$ is a higher-order  term when compared to $I_1$ and $I_3$. Thus 
	 we can ensure that $W_{r r} \cdot W_{n n}-\left|W_{r n}\right|^2 \geq  (1-\delta)I_2 > 0$, which also satisfies the convex condition in Lemma \ref{lemma:eigen}. Then we derive 
	\begin{align*}
		W_{rr}\cdot W_{nn}-|W_{rn}|^2
		&\geq  (1-\delta)I_2=(1-\delta)\frac{8(b-1)}{a^2b^3}\cdot|W|^{2-3b}\cdot\big(\frac{x_n}{\xi}\big)^{\frac{4}{a}-2}\cdot\xi^{-2}\\
		&\geq (1-\delta) \frac{8(b-1)}{a^2b^3}\cdot|W|^{2-3b}\cdot |W|^{2b-ab}\cdot\xi^{-2}\\
		&=	\widetilde{C}_1(a,b,\delta)\cdot|W|^{2-b-ab}\cdot\xi^{-2}.
	\end{align*}
In view of  \eqref{boundofr2}, we obtain 
	\begin{equation*}
		W_{r r} \in\left[\frac{2}{b} \cdot|W|^{1-b},\left(\frac{2}{b}+\frac{4(b-1)}{b^2} \delta(1-\delta)^{-1}\right) \cdot|W|^{1-b}\right] .
	\end{equation*}
By \eqref{boundofxneps}, we have
	\begin{equation*}
		W_{n n} \in\left[\frac{2ab-4}{a^2b^2}\cdot|W|^{1-a b} \cdot \xi^{-2}, \left(\frac{2(a-2)}{a^2 b}(1-\delta)^{a-1}+  \frac{4(b-1)}{a^2 b^2}(1-\delta)^{a-2}\right)\cdot|W|^{1-a b} \cdot \xi^{-2}\right],
	\end{equation*}
	where  the following  coefficient is positive 
	\begin{align*}
	&\ \ \ \ 	\frac{2(a-2)}{a^2 b}(1-\delta)^{a-1}+  \frac{4(b-1)}{a^2 b^2}(1-\delta)^{a-2} \\
	&>-\frac{4(b-1)}{a^2 b^2}(1-\delta)^{a-1}+  \frac{4(b-1)}{a^2 b^2}(1-\delta)^{a-2}\\
		&= (1-\delta)^{a-1}\left(\left((1-\delta)^{-1}-1\right) \cdot \frac{4(b-1)}{a^2b^2}\right) \\
		&> 0
	\end{align*}
based on the fact that 
\[\frac{2(a-2)}{a^2 b}+  \frac{4(b-1)}{a^2 b^2}=\frac{2ab-4}{a^2b^2}>0.\]
It follows that 
	\begin{equation*}
	W_{r r} \cdot W_{n n}-\left|W_{r n}\right|^2 \leq  W_{r r} \cdot W_{n n} \leq  \widetilde{C}_{2}(a, b, \delta) \cdot|W|^{2-b-a b} \cdot \xi^{-2}.
\end{equation*}
	Since $2-\frac{2}{a}>0$ and $\xi<1$ is sufficient small, we can infer 
		\begin{align*}
		&\ \ \ \ 	W_{r r}+W_{n n}\\
		 & \leq \left(\frac{2}{b}+\frac{4(b-1)}{b^2} \delta(1-\delta)^{-1}\right) \cdot|W|^{1-b}+\left(\frac{2(a-2)}{a^2 b}(1-\delta)^{a-1}+  \frac{4(b-1)}{a^2 b^2}(1-\delta)^{a-2}\right) \cdot|W|^{1-a b} \cdot \xi^{-2} \\
			& =\left(\left(\frac{2}{b}+\frac{4(b-1)}{b^2} \delta(1-\delta)^{-1}\right) \cdot|W|^{a b-b} \cdot \xi^2+\frac{2(a-2)}{a^2 b}(1-\delta)^{a-1}+  \frac{4(b-1)}{a^2 b^2}(1-\delta)^{a-2}\right) \cdot|W|^{1-a b} \cdot \xi^{-2} \\
			& \leq \left(\left(\frac{2}{b}+\frac{4(b-1)}{b^2} \delta(1-\delta)^{-1}\right) \cdot(\operatorname{diam}(\varOmega))^{2-\frac{2}{a}} \cdot \xi^{\frac{2}{a}}+\frac{2(a-2)}{a^2 b}(1-\delta)^{a-1}+  \frac{4(b-1)}{a^2 b^2}(1-\delta)^{a-2} \right) \cdot|W|^{1-a b} \cdot \xi^{-2} \\
			& \leq \widetilde{C}_{3}(a,b,\delta,\operatorname{diam}(\varOmega)) \cdot|W|^{1-a b} \cdot \xi^{-2}.
		\end{align*}

\end{document}